\documentclass{cm-l}

\usepackage{amssymb}
\usepackage{amsmath}
\usepackage{float}
\usepackage{mathtools}
\usepackage{hyperref} 
\usepackage{rotating}
\usepackage{graphicx}
\usepackage{mathrsfs}
\DeclarePairedDelimiter\abs{\lvert}{\rvert}
\newtheorem{example}{Example}[section]
\newtheorem{theorem}{Theorem}[section]
\newtheorem{lemma}[theorem]{Lemma}
\newtheorem{definition}[theorem]{Definition}
\newtheorem{proposition}[theorem]{Proposition}
\newtheorem{corollary}[theorem]{Corollary}
\theoremstyle{remark}
\newtheorem{remark}[theorem]{Remark}

\usepackage[cmtip,all]{xy}
\usepackage{tikz-cd}

\theoremstyle{definition}

\theoremstyle{remark}

\numberwithin{equation}{section}

\begin{document}

\title[LEFT-SYMMETRIC STRUCTURES ON $\mathfrak{osp}(1|2)$]{ON THE NON-EXISTENCE OF A LEFT-SYMMETRIC STRUCTURE ON $\mathfrak{osp}(1|2)$ IN ANY CHARACTERISTIC}


%    author one information
\author[J. Loza]{Josemaria Loza}
\address{Dubai, United Arab Emirates}
\curraddr{New York University Abu Dhabi
PO Box 129188, Saadiyat Island, Abu Dhabi, United Arab Emirates}
\email{jl13583@nyu.edu}
\thanks{}

%    author two information
% \author[S. Bouarroudj]{Sofiane Bouarroudj}
% \address{New York University Abu Dhabi
% PO Box 129188, Saadiyat Island, Abu Dhabi, United Arab Emirates}
% \curraddr{New York University Abu Dhabi
% PO Box 129188, Saadiyat Island, Abu Dhabi, United Arab Emirates}
% \email{sofiane.bouarroudj@nyu.edu}
% \thanks{}

%    \subjclass is required.


\date{3 August 2026}

\dedicatory{}

%    Abstract is required.
\begin{abstract} Burde proved that the Lie algebra $\mathfrak{sl}_2(\mathbb{K})$, over a field $\mathbb{K}$ of characteristic $p > 2$, admits compatible left-symmetric algebra structures if and only if $p = 3$. Since the even part of the Lie superalgebra $\mathfrak{osp}(1|2)$ is isomorphic to $\mathfrak{sl}_2(\mathbb{K})$, we ask whether Burde's structures can be extended to compatible left-symmetric superalgebra structures on $\mathfrak{osp}(1|2)$. We verify Burde's left-symmetric structures on $\mathfrak{sl}_2(\mathbb{K})$, and then test their possible extensions to $\mathfrak{osp}(1|2)$. We show that no such extension exists in characteristic 3. Together with the obstruction coming from the even part for characteristic $0$ and $p > 3$, this gives non-existence for a field $\mathbb{K}$ of characteristic different from $2$.
\end{abstract}

\maketitle


\section{Introduction} \label{Introduction}
A left-symmetric algebra (LSA), also known as a pre-Lie algebra, is a non-associative algebra over a vector space $\mathcal{A}$ such that its associator
\[
    (a, b, c) = (ab)c - a(bc), \quad \forall a, b, c \in \mathcal{A}
\]

\noindent
is symmetric in the first two arguments. Every LSA naturally gives rise to a Lie algebra through the commutator $[a, b] \coloneqq ab-ba$, see Definition \ref{lsa-def}. A fundamental problem in the study of LSAs and their superization, left-symmetric superalgebras (LSSAs), is determining which Lie (super)algebras are compatible with left-symmetric structures. It is known that a Lie algebra $\mathfrak{g}$ over a field $\mathbb{K}$ of characteristic $0$  with $[\mathfrak{g}, \mathfrak{g}]=\mathfrak{g}$ does not admit any left-symmetric structures \cite{helmstetter}. In particular, semisimple Lie algebras over a field of characteristic $0$ do not admit LSAs. This is no longer true for Lie superalgebras. See \cite{dimitrov-zhang} for more information. Nevertheless, different phenomena arise over fields of positive characteristic. For instance, Burde proved in \cite{burde1} that the simple Lie algebra $\mathfrak{sl}_2(\mathbb{K})$ over a field $\mathbb{K}$ of characteristic $p > 2$ admits non-trivial left-symmetric structures if and only if $p = 3$. Specifically, when $\mathbb{K}$ is of characteristic $3$, Burde proved that there exist four non-isomorphic structures compatible with $\mathfrak{sl}_2(\mathbb{K})$.

Several authors have recently studied the existence of LSSAs on Lie superalgebras. Dimitrov and Zhang in \cite{dimitrov-zhang} studied the existence and classification of LSSAs on special linear Lie superalgebras $\mathfrak{sl}(m | n)$ with $m \neq n$. Specifically, they found the following: (i) a complete classification of the left-symmetric superalgebras on $\mathfrak{sl}(2|1)$, (ii) $\mathfrak{sl}(m|1)$ does not admit left-symmetric superalgebras for $m \geq 3$, and (iii) $\mathfrak{sl}(m+1 |m)$ admits a left-symmetric superalgebra for every $m \geq 1$. Lastly, they conjecture that $\mathfrak{sl}(m|n)$ admits left-symmetric superalgebras if and only if $m = n +1$. 

Similarly, Bouarroudj and Radu in \cite{bouarroudj-radu} investigated all four-dimensional Lie superalgebras that can be obtained as Lagrangian extensions, and the LSSAs compatible with these.

The orthosymplectic Lie superalgebra $\mathfrak{osp}(1|2)$ is a classical Lie superalgebra. Its even part $\mathfrak{osp}(1|2)_{\bar0}$ is isomorphic to $\mathfrak{sl}_2(\mathbb{K})$, while its odd part $\mathfrak{osp}(1|2)_{\bar1}$ is the standard irreducible 2-dimensional $\mathfrak{sl}_2(\mathbb{K})$-module. That is,
\[
    \mathfrak{osp}(1|2) = \mathfrak{sl}_2(\mathbb{K}) \oplus \mathfrak{osp}(1|2)_{\bar1}.
\]
\noindent
Given this relationship, one may ask whether Burde's LSAs on $\mathfrak{sl}_2(\mathbb{K})$ can be extended to compatible LSSAs on $\mathfrak{osp}(1|2)$. This paper explores this question by adapting Burde's methods through an algebraic formulation. For a geometric interpretation of similar structures, see \cite{bouarroudj-radu}. We do so specifically by first introducing coefficient structures for the different relations among generators in the basis of $\mathfrak{osp}(1|2)$. Using these structures, we test for such LSSAs through a two-stage process: 
\begin{itemize}
    \setlength{\itemsep}{0pt}   % Space between items
    \setlength{\parskip}{0pt}   % Space between paragraphs in items
    \setlength{\topsep}{0pt}
    \item [(i)] Stage 1: Recover the basis relations in $\mathfrak{osp}(1|2)$ using the supercommutator associated with the left multiplication, defined in \cite{burde1};
    \item [(ii)] Stage 2: Satisfy the $\mathbb{Z}/2\mathbb{Z}$-graded left-symmetric identity through the associator.
\end{itemize} 
\noindent
Stage 2, in particular, entails solving multivariate systems of nonlinear polynomial equations using two approaches central to algebraic geometry: 1) Gröbner bases, and 2) polynomial factorization. We find that Stage 1 is readily satisfied. Nevertheless, the resulting polynomial systems in Stage 2 are inconsistent. That is, the graded left-symmetric identity is not satisfied. The Gröbner basis computation certifies this inconsistency by producing $G = \{1\}$, and the factorization analysis provides a complementary check. Given that both approaches in Stage $2$ fail to find a solution to the polynomial system of equations, we conclude that the Lie superalgebra $\mathfrak{osp}(1|2)$ does not admit LSSAs over a field $\mathbb{K}$ of characteristic different from 2.

The paper begins in Section \ref{Preliminaries} with definitions and key properties of Lie algebras, Lie superalgebras, left-symmetric structures, and Gröbner bases. We also formalize notation and recall the definition of the Lie superalgebra $\mathfrak{osp}(1|2)$ as a superalgebra preserving a non-degenerate even supersymmetric bilinear form on the superspace $\mathbb{K}^{1|2}$ \cite{scheunert}. Section \ref{LSS on sl2(k)} outlines a strategy for constructing LSSAs on $\mathfrak{osp}(1|2)$ by adapting Burde's method to the superized setting. We verify Burde's LSAs on $\mathfrak{sl}_{2}(\mathbb{K})$ and explicitly represent them in terms of the $\langle E, F, H \rangle$ basis. Lastly, in Section \ref{Superization LSS} we test possible extensions of Burde's structures on $\mathfrak{osp}(1|2)$ by attempting Stages 1 and 2. Further exploration could aim to identify the structural reason behind this non-existence using other mathematical frameworks, namely, a superized cohomological approach. 

Throughout this paper, we assume $\mathrm{char}(\mathbb{K})$ is not equal to $2$. Lie superalgebras in characteristic $2$ are exceptional for two reasons. First, the concept of parity and the distinction between symmetric and skew-symmetric vanish because $-1 = 1$. Second, the Lie algebra $\mathfrak{sl}_{2}(\mathbb{K})$ is not simple in characteristic $2$, as $H = \mathrm{diag}(1, -1) = \mathrm{diag}(1,1)$ produces a non-trivial center. The case of Lie superalgebras in characteristic $2$ has been treated separately by Bouarroudj et al. in \cite{BENAYADI2025108086}. The accompanying Python code for Sections \ref{LSS on sl2(k)} and \ref{Superization LSS} can be found in the GitHub repository \href{https://github.com/jl13583/capstone_josemaria}{capstone\textunderscore josemaria}. 

\section[Preliminaries]{Preliminaries}\label{Preliminaries}

The purpose of this section is to recall the main definitions and theorems relevant to this paper.

\subsection{Background on Lie Superalgebras.} \label{Lie superalgebras}

Let $V=V_{\bar 0}\oplus V_{\bar 1}$ be a superspace defined over  ${\mathbb K}$. The parity of a homogeneous element $v\in V_{\bar{i}}$ is denoted by $|v| \coloneqq \bar{i}$. The element $v$ is called \textit{even} if $v\in V_{\bar 0}$ and \textit{odd} if $v\in V_{\bar 1}$. The {\it superdimension} of $V$ is $\mathrm{sdim}(V)=a+b\epsilon$, where $\epsilon^2=1$, and $a=\mathrm{dim}(V_{\bar 0})$, $b=\mathrm{dim}(V_{\bar 1})$. Usually, $\mathrm{sdim}(V)$ is abbreviated as $a \mid b$. 

We may identify \( V \cong \mathbb{K}^{m|n} \), the superspace of dimension \( m \mid n \) over a field \( \mathbb{K} \).
\begin{definition}[Superalgebra]
    \normalfont
    A superalgebra is a vector superspace $V=V_{\bar0} \oplus V_{\bar1}$ endowed with a bilinear multiplication $\ \cdot \ $ that satisfies $V_{i} \cdot V_{j} \subseteq V_{i + j \bmod 2}$, $i, j \in \mathbb{Z}/2\mathbb{Z}$.  
\end{definition}

\begin{definition}[Lie superalgebra] \label{def-lie-super}
    \normalfont
    A \textbf{Lie superalgebra} is a $\mathfrak{g} = \mathfrak{g}_{\bar0} \oplus \mathfrak{g}_{\bar1}$ superalgebra over a field $\mathbb{K}$ of characteristic $p \not= 2$ equipped with a bilinear operation $[\cdot \mathbin{,} \cdot]$, the Lie superbracket
    \[
        [\cdot \mathbin{,} \cdot] \colon \mathfrak{g} \times \mathfrak{g} \to \mathfrak{g} ,\ (a,b) \mapsto [a,b]
    \]

    \noindent satisfying the following properties (for any homogeneous elements $a, b, c \in \mathfrak{g}$):

    \begin{enumerate}
        \setlength{\itemsep}{0pt}   % Space between items
        \setlength{\parskip}{0pt}   % Space between paragraphs in items
        \setlength{\topsep}{0pt}
        \item [(i)] Skew-supersymmetry: $[a, b] = -(-1)^{\abs{a}\abs{b}}[b,a]$;
        \item [(ii)] Super Jacobi identity: $[a, [b,c]] = [[a,b], c] + (-1)^{\abs{a}\abs{b}}[b, [a, c]]$;
        \item [(iii)] In the case $p=3$ only, one must add the condition $[a,[a,a]]=0,~\forall a\in \mathfrak{g}_{\bar1}$.
    \end{enumerate}
\end{definition}

\begin{remark}
    Note that if $\mathfrak{g} = \mathfrak{g}_{\bar0}$, we then recover the definition of a Lie algebra (see Remark 2.16).
\end{remark}

The concepts of \textit{subalgebras} and \textit{ideals} are understood in a similar way to Lie algebras with $\mathbb{Z}/2\mathbb{Z}$-grading. However, in the super setting, ideals can be homogeneous or not.   

\begin{example}
    \normalfont
    Let $\mathcal{A}=\mathcal{A}_{\bar0} \oplus \mathcal{A}_{\bar1}$ be an associative superalgebra. $\mathcal{A}$ can be turned into a Lie superalgebra by defining a new multiplication $[\cdot \mathbin{,} \cdot] \colon \mathcal{A} \times \mathcal{A} \to \mathcal{A}$ such that 

    \begin{equation}
        [a, b] = ab-(-1)^{\abs{a}\abs{b}}ba
        \label{supercommutator}
    \end{equation}

    \noindent
    for all homogeneous $a, b \in \mathcal{A}$. 
    \label{example-assoc-superalgebra}
\end{example}

As with Lie algebras, homomorphisms can also be defined. Let $\mathfrak{g}_1= \mathfrak{g}_{1, {\bar0}} \oplus \mathfrak{g}_{1,{\bar1}}$ and $\mathfrak{g}_2= \mathfrak{g}_{2, {\bar0}} \oplus \mathfrak{g}_{2,{\bar1}}$ be two Lie superalgebras. Then, a \textit{homomorphism of Lie superalgebras} is an even linear map $\phi \colon \mathfrak{g}_1 \to \mathfrak{g}_2$ satisfying
\[
    \phi([a,b])=[\phi(a), \phi(b)], \quad \forall a, b \in \mathfrak{g}_1.
\]

By an even map, we mean that $\phi$ preserves parity, $\phi \in \mathrm{Hom}_{\bar0}(\mathfrak{g}_1, \mathfrak{g}_2)$, or, equivalently, that $\phi(\mathfrak{g}_{1,{\bar0}}) \subset \mathfrak{g}_{2,{\bar0}}$ and $\phi(\mathfrak{g}_{1,{\bar1}}) \subset \mathfrak{g}_{2,{\bar1}}$.

\begin{definition}[$\mathfrak{g}$-module]
    \normalfont
    If $\mathfrak{g}$ is a Lie superalgebra, then a $\mathfrak{g}$\textit{-module} is a vector space $V$ together with a bilinear map $\mathfrak{g} \times V \to V, \ (a, v) \mapsto av$ satisfying
    \[
       [a, b]v = a(bv) - (-1)^{\abs{a} \abs{b}}b(av),
    \]

    \noindent
    for all $a, b \in \mathfrak{g}$ and $v \in V$.
\end{definition}

\begin{example}
    \normalfont
    Let $\mathfrak{g}$ be a Lie superalgebra. Then, the space of all endomorphisms of $\mathfrak{g}$, $\mathrm{End}(\mathfrak{g})$, is clearly associative, and is thus a Lie superalgebra by Example \ref{example-assoc-superalgebra}. Now define the \textit{adjoint representation} of $\mathfrak{g}$ by $\mathrm{ad} \colon \mathfrak{g} \to \mathrm{End}(\mathfrak{g})$ such that 
    \[
        \mathrm{ad}_a(b)=[a, b], \quad \forall a, b \in \mathfrak{g}.
    \]

    We notice that $\mathrm{ad}$ is a homomorphism of Lie superalgebras due to the Super Jacobi identity of Definition \ref{def-lie-super}. Moreover, the resulting action $\mathrm{ad}_a$ of $\mathfrak{g}$ onto itself is called the \textit{adjoint action}, similar to the adjoint module of a Lie algebra. 
\end{example}

\begin{remark}
    The even part of $\mathfrak{g}$ is a Lie algebra on its own. This is because, by the $\mathbb{Z}/2\mathbb{Z}$-grading, we have closure under the Lie superbracket; that is, $[\mathfrak{g}_{\bar0}, \mathfrak{g}_{\bar0}] \subset \mathfrak{g}_{\bar0}$. Thus, in $\mathfrak{g}_{\bar0}$, the superbracket reduces to the Lie bracket. This is not true for the odd part $\mathfrak{g}_{\bar1}$ as closure is not satisfied: $[\mathfrak{g}_{\bar1}, \mathfrak{g}_{\bar1}] \subset \mathfrak{g}_{\bar0}$. Under the adjoint action, $\mathfrak{g}_{\bar1}$ is actually a $\mathfrak{g}_{\bar0}$-module. For more details, see \cite{cheng}. 
\end{remark}

\begin{definition}[Superderivations]
    \normalfont
    Let $\mathcal{A}=\mathcal{A}_{\bar0} \oplus \mathcal{A}_{\bar1}$ be a superalgebra. Then a linear endomorphism $D \in \mathrm{End}(\mathcal{A})_s$ such that $s \in \mathbb{Z}/2\mathbb{Z}$ is called a superderivation of $\mathcal{A}$ of degree $s$ if 
    \[
        D(ab)=D(a)b +(-1)^{s \abs{a}}aD(b), \quad \forall a, b \in \mathcal{A}.
    \]

    Let $\mathrm{Der}(\mathcal{A})_s$ denote the space of derivations on $\mathcal{A}$ of degree $s$. If $s = \bar0$, the subspace $\mathrm{Der}_{\bar0}(\mathcal{A})$ consists of all even derivations, meaning that $D(\mathcal{A}_s) \subset \mathcal{A}_s$ for $s \in \mathbb{Z}/2\mathbb{Z}$. Conversely, if $s = \bar1$, the subspace $\mathrm{Der}_{\bar1}(\mathcal{A})$ consists of all odd derivations, implying $D(\mathcal{A}_s) \subset \mathcal{A}_{s+1 \bmod 2}$ for $s \in \mathbb{Z}/2\mathbb{Z}$. Hence, the superspace of superderivations of $\mathcal{A}$, $\mathrm{Der}(\mathcal{A}) = \mathrm{Der}(\mathcal{A})_{\bar0} \oplus \mathrm{Der}(\mathcal{A})_{\bar1}$ is called the \textit{Lie superalgebra of superderivations of $\mathcal{A}$} and is a subalgebra of the Lie superalgebra $\mathrm{End}(\mathcal{A})$ paired with the superbracket operator. For more details on superderivations, see \cite{musson, cheng}.
\end{definition}

\begin{example}
    \normalfont
    Take the adjoint action. If $\mathfrak{g}$ is a Lie superalgebra, then similarly to Example 2.1, through the super Jacobi identity, $\mathrm{ad \ \mathfrak{g} \in \mathrm{Der}(\mathfrak{g})}$. That is, $\mathrm{ad}_a$ is a superderivation, for all $a \in \mathfrak{g}$. These are called the \textit{inner derivations}, and they form an ideal in $\mathrm{Der}(\mathfrak{g})$. 
\end{example}

Similarly to Lie algebras, we can define subspaces, ideals, and related concepts. Let $\mathfrak{g}$ be a Lie superalgebra. If $V, W$ are subspaces of $\mathfrak{g}$, then $[V, W]= \{[v, w] \mid v \in v, \ w \in w\}$ is a subspace of $\mathfrak{g}$ as well. If $\mathfrak{i}$ is a $\mathbb{Z}/2\mathbb{Z}$-graded subspace of $\mathfrak{g}$, then $\mathfrak{i}$ is a graded \textit{ideal} if $[\mathfrak{i}, \mathfrak{g}] \subset \mathfrak{i}$.  

We call $\mathfrak{g}^{(i)}$ the \textit{i}-th derived superalgebra. The \textit{derived series} $\mathfrak{g}^{(n)}$ of $\mathfrak{g}$ is defined by 
\[
    \mathfrak{g}^{(0)} = \mathfrak{g}, \quad \mathfrak{g}^{(n+1)}=[\mathfrak{g}^{(n)}, \mathfrak{g}^{(n)}], \quad n\geq 0.
\]

The \textit{lower central series} $\mathfrak{g}^{[n]}$ of $\mathfrak{g}$ is defined by 
\[
     \mathfrak{g}^{[0]} = \mathfrak{g}, \quad \mathfrak{g}^{[n+1]}=[\mathfrak{g}, \mathfrak{g}^{[n]}], \quad n\geq 0.
\]

The Lie superalgebra $\mathfrak{g} $ is said to be \textit{solvable} if $\mathfrak{g}^{(n)}=0$ for some $n \geq 1$. Likewise, $\mathfrak{g}$ is said to be \textit{nilpotent} if $\mathfrak{g}^{[n]}=0$ for some $n \geq 1$. Lastly, $\mathfrak{g}$ is said to be \textit{abelian} if $[\mathfrak{g}, \mathfrak{g}]=0$. 

A Lie superalgebra $\mathfrak{g}$ is \textit{simple} if it is not abelian and the only $\mathbb{Z}/2\mathbb{Z}$-graded ideals are the trivial ones ($0$ and $\mathfrak{g}$ itself). 

\begin{example}
    \normalfont
    Let $V=V_{\bar0} \oplus V_{\bar1} \cong \mathbb{C}^{m \mid n}$ be a vector superspace with $\mathrm{sdim}=m \mid n$ such that $\mathrm{End}(V)$ is an associative superalgebra. By Example \ref{example-assoc-superalgebra}, equipping $\mathrm{End}(V)$ with the Lie superbracket (in this case called the \textit{supercommutator}) transforms it into a Lie superalgebra called the \textbf{general linear Lie superalgebra} and is denoted by $\mathfrak{gl}(V)$ or, in this case, by $\mathfrak{gl}(m|n)$. By fixing a basis, $\mathfrak{gl}(m|n)$ can be realized as the set of all $(m +n) \times (m +n)$ block matrices of the form 
    \begin{equation}
        X = \begin{pmatrix} A & B \\ C & D\end{pmatrix}
        \label{blockmat},
    \end{equation}

    \noindent
    where $A, B, C,$ and $D$ are matrices respectively of the form $m \times m, \ m \times n, \ n \times m,$ and $n \times n$. The even subalgebra $\mathfrak{gl}(m|n)_{\bar{0}}$ consists of all matrices in the form of Eq. \eqref{blockmat} such that $B= C = O$, while the odd subalgebra $\mathfrak{gl}(m|n)_{\bar{1}}$ consists of all matrices of the form in Eq. \eqref{blockmat} with $A= D = O$. 

    Define the \textit{supertrace} of $X$ by 
    \[
        \mathrm{str}(X) = \mathrm{tr}(A) - \mathrm{tr}(D),
    \]
    \noindent 
    where $\mathrm{tr}$ denotes the usual  trace of a square matrix. 

    One can check that for any $X, Y \in \mathfrak{gl}(m|n)$, we have $\mathrm{str([X, Y])=0}$ due to the cyclic property of the supertrace operator. See \cite{cheng, musson} for more details.
    
    The subalgebra of $\mathfrak{gl}(m|n)$ consisting of all matrices in the form of Eq. \eqref{blockmat} with supertrace zero is called the \textbf{special linear Lie superalgebra} and is denoted by $\mathfrak{sl}(m|n)$. By using the cyclic property of the supertrace and the fact that $\mathfrak{sl}(m|n)$ is spanned by commutators $[X, Y] \in \mathfrak{gl}(m|n)$, we can show $[\mathfrak{gl}(m|n), \mathfrak{gl}(m|n)] = \mathfrak{sl}(m|n)$. 
\end{example}

Now, we will briefly define bilinear forms, useful for understanding the subsequent definition of the Lie superalgebra $\mathfrak{osp}(1|2)$. For more details, see \cite{reich}.

\begin{definition}[Bilinear form]
    \normalfont
    Let $V$ be a vector space over a field $\mathbb{K}$. A \textit{bilinear form} is a bilinear map $\mathscr{B} \colon V \times V \to \mathbb{K}$ such that for all $u, v, w \in V$ and $\alpha, \beta \in \mathbb{K}$, the following are satisfied

    \begin{enumerate}
        \setlength{\itemsep}{0pt}   % Space between items
        \setlength{\parskip}{0pt}   % Space between paragraphs in items
        \setlength{\topsep}{0pt}
        \item [(i)] $\mathscr{B}(\alpha u +\beta v, w)= \alpha \mathscr{B}(u, w) + \beta \mathscr{B}(v, w)$,
        \item [(ii)] $\mathscr{B}(u, \alpha v +\beta w)= \alpha \mathscr{B}(u, w) + \beta \mathscr{B}(v, w)$.
    \end{enumerate}

    Let $\mathscr{B}$ be a bilinear form. Then, $\mathscr{B}$ is said to be \textit{symmetric} if $\mathscr{B}(u, v)=\mathscr{B}(v, u)$ for all $u, v \in V$. $\mathscr{B}$ is said to be \textit{antisymmetric} if $\mathscr{B}(u, v)=-\mathscr{B}(v,u)$ for all $u, v \in V$. Moreover, $\mathscr{B}$ is called \textit{alternating} if $\mathscr{B}(u, u)=0$ for all $u \in V$. 

    Relevant to this paper, a bilinear form $\mathscr{B}$ is called \textit{nondegenerate} if for every $v \in V$ with $v \neq 0$, there exists $u \in V$ such that 
    \[
        \mathscr{B}(u, v) \not= 0. 
    \]
\end{definition}\label{bilinear-form}

    The definition of bilinear forms can be easily extended to Lie superalgebras by incorporating the $\mathbb{Z}/2\mathbb{Z}$-grading. If we have that $V=V_{\bar0} \oplus V_{\bar1}$ is a vector superspace, then $\mathscr{B}$ is a bilinear form if it satisfies conditions (i) and (ii) from Definition \ref{bilinear-form}, for all homogeneous $u, v, w \in V$. A bilinear form $\mathscr{B}$ is called \textit{even} (respectively, \textit{odd}) if $\mathscr{B}(V_i, V_j)=0$ unless $i+j= \bar{0}$ (respectively, $i+j= \bar{1}$). 

    A bilinear form $\mathscr{B}$ over a vector superspace $V$ is called \textit{supersymmetric} if for all homogeneous $u, v \in V$
    \[
        \mathscr{B}(u, v)=(-1)^{\abs{u} \abs{v}}(v, u).
    \]

    \noindent
     The bilinear form $\mathscr{B}$ is called \textit{anti-supersymmetric} if for all homogeneous $u, v \in V$ 
     \[
        \mathscr{B}(u, v)=-(-1)^{\abs{u} \abs{v}}(v, u).
     \] 
     The concepts of alternating and nondegeneracy remain the same as in vector spaces. 

\begin{definition}[Orthosymplectic Lie superalgebra]
    \normalfont
    Let $\mathscr{B}$ be a nondegenerate, even bilinear form on a vector superspace $V$ that is either supersymmetric or anti-supersymmetric. Then, $\mathfrak{osp}(V, \mathscr{B})$ denotes the \textit{orthosymplectic Lie superalgebra} preserving the bilinear form $\mathscr{B}$ and is defined by
    \[
       \mathfrak{osp}(V, \mathscr{B}) = \{X \in \mathfrak{gl}(V) \mid \mathscr{B}(X(v), w) + (-1)^{|X||v|} \mathscr{B}(v, X(w)) = 0 , \forall v,w \in V\}.
    \]
    
    We can also define this Lie superalgebra in matrix form. Fix a homogeneous basis for the superspace $V = V_{\bar{0}} \oplus V_{\bar{1}}$ with $\mathrm{dim}(V_{\bar{0}}) = m$ and $\mathrm{dim}(V_{\bar{1}}) = n$ identifying the isomorphism $V \cong \mathbb{K}^{m|n}$. Let
    \[
        J = \begin{pmatrix} G & O \\ O & H \end{pmatrix}
    \]

    \noindent
    be the matrix representing $\mathscr{B}$ in block form such that $G$ is an $m \times m$ matrix and $H$ is an $n \times n$ matrix; additionally, $G$ is symmetric while $H$ is antisymmetric. A block matrix $X \in \mathfrak{gl}(m | n)$ in the form of Eq. \eqref{blockmat} lies in $\mathfrak{osp}(V, \mathscr{B}) = \mathfrak{osp}(m|n)$ if 
    \[
        A^{T}G + GA= B^{T}G - HC = D^{T}H + HD = 0.
    \]

    For more details on  orthosymplectic Lie algebras, see \cite{scheunert, musson}.
    
    The Lie superalgebra $\mathfrak{osp}(1|2) = \mathfrak{osp}(1|2)_{\bar0} \oplus \mathfrak{osp}(1|2)_{\bar1}$ is a classical Lie superalgebra of superdimension $3|2$, defined as the subalgebra of $\mathfrak{gl}(1|2)$ preserving a non-degenerate even supersymmetric bilinear form $\mathscr{B}$ on the superspace $\mathbb{K}^{1|2}$. Concretely, we have that
    \[
        \mathfrak{osp}(1|2) = \{ X \in \mathfrak{gl}(1|2) \mid \mathscr{B}(Xv, w) + (-1)^{|X||v|} \mathscr{B}(v, Xw) = 0, \forall v,w \in \mathbb{K}^{1|2} \}.
    \]

    \noindent
    We can take $\mathscr{B}$ to be represented by the matrix

    \begin{equation}
        J = 
        \begin{pmatrix}
        1 & 0 & 0 \\
        0 & 0 & 1 \\
        0 & -1 & 0
        \end{pmatrix},  
        \label{biformmat}
    \end{equation}

    \noindent
    with respect to a homogeneous basis, where $G=1$ and $H= \begin{pmatrix} 0 & 1 \\ -1 & 0 \end{pmatrix}$. Then
    \[
        \mathfrak{osp}(1|2) = \{ X \in \mathfrak{gl}(1|2) \mid X^\text{sT}J + JX = 0 \},
    \]
    where \(X^\text{sT}\) denotes the supertranspose of $X$. If $X$ is in the form of Eq. \eqref{blockmat}, then 

    \[
        X^{sT} = \begin{pmatrix} A^T & C^T \\ -B^T & D^T \end{pmatrix}.
    \]
    
    Explicitly, the even part $\mathfrak{osp}(1|2)_{\bar0} \cong \mathfrak{sl}_2(\mathbb{K})$, and the odd part $\mathfrak{osp}(1|2)_{\bar1}$ transforms to a 2-dimensional module under the adjoint action of \(\mathfrak{sl}_2(\mathbb{K})\). Thus,
    \begin{equation}
        \mathfrak{osp}(1|2) = \mathfrak{sl}_2(\mathbb{K}) \oplus \mathfrak{osp}(1|2)_{\bar1}.
        \label{osp12}
    \end{equation}

    \noindent
    This isomorphism is exceptional and only true for $\mathfrak{osp}(1|2)$.
    This makes $\mathfrak{osp}(1|2)$ a natural example of a simple Lie superalgebra, with tractable structure and close connections to Lie algebras. In particular, the isomorphism $\mathfrak{osp}(1|2)_{\bar 0} \cong \mathfrak{sl}_2(\mathbb{K})$ implies that the (bosonic) generators span the even part
    \begin{equation}
        \mathfrak{osp}(1|2)_{\bar{0}} = \Biggl\langle E = \begin{pmatrix} 0 & 1 & 0 \\ 0 & 0 & 0 \\ 0 & 0 & 0 \end{pmatrix}, \ F = \begin{pmatrix} 0 & 0 & 0 \\ 1 & 0 & 0 \\ 0 & 0 & 0 \end{pmatrix}, \ H = \begin{pmatrix} 1 & 0 & 0 \\ 0 & -1 & 0 \\ 0 & 0 & 0 \end{pmatrix} \Biggr\rangle.
        \label{evengen}
    \end{equation} 

    \noindent
    These are called the \textit{even generators}, and they satisfy the following commutation relations 
    \begin{equation}
    \label{sl2-relations}
    [H, E]=2E, \ [H,F] = -2F, \ \mathrm{and} \ [E,F] = H. 
    \end{equation}
    \noindent
    This also implies that $\mathrm{dim}(\mathfrak{osp}(1|2)_{\bar{0}}) = \mathrm{dim}(\mathfrak{sl}_2(\mathbb{K}))=3$.
    
    It is important not to confuse the matrix $H$ in Eq. \eqref{biformmat} with that of Eq. \eqref{evengen}.
    
    Similarly, the (fermionic) generators span the odd part $\mathfrak{osp}(1|2)_{\bar1}$
    
    \begin{equation}
        \mathfrak{osp}(1|2)_{\bar{1}} = \Biggl\langle T_{+} = \begin{pmatrix} 0 & 0 & 1 \\ 0 & 0 & 0 \\ 0 & 1 & 0 \end{pmatrix}, \ T_{-} = \begin{pmatrix} 0 & 0 & 0 \\ 0 & 0 & -1 \\ 1 & 0 & 0 \end{pmatrix} \Biggr\rangle. \label{oddgen}
    \end{equation}

    \noindent
    These are called the \textit{odd generators} and, by their parity, they satisfy the following anticommutation relations:
    \begin{equation}
        \{T_{+}, T_{+}\} = 2E, \ \{T_{+}, T_{-}\} = \{T_{-}, T_{+}\} = H, \ \{T_{-}, T_{-}\} = -2F. \label{odd_rela}
    \end{equation}

    Now, considering the odd-even relations:
    \begin{equation}
        [T_{+}, E] = [T_{-}, F] = 0, \ [T_{+}, F] = [T_{-}, H] = T_{-}, \ [T_{+}, H] = -T_{+}, \ [T_{-}, E] = T_{+}. \label{even_odd_rela}
    \end{equation}

    Thus, $\mathfrak{osp}(1|2)$ is said to have a (graded) basis denoted by $\langle E,F,H \mid T_{+}, T_{-}\rangle$.
    
    \begin{remark}
        In physics, a normalization of $\pm \frac{1}{2}$ is applied to the commutation relations with $H$ to assign a spin to odd generators. In mathematics, a normalization of $2$ is assigned as a result of the even subalgebra relations involving $H$. See \cite{mnev}. 
    \end{remark}
\end{definition}

\subsection{Non-Associative Algebras and Left-Symmetric (Super)Algebras.} \label{LSAs}

This section will briefly cover the definitions and properties of non-associative algebras and left-symmetric structures. 

\begin{definition}[Left-symmetric algebras] \label{lsa-def}
    \normalfont
    Let $\mathcal{A}$ be a vector space over a field $\mathbb{K}$ equipped with a bilinear product $(a, b) \mapsto ab$. $\mathcal{A}$ is called a \textit{left-symmetric algebra} (or LSA, for short) if for all $a, b, c \in \mathcal{A}$, the associator 

    \begin{equation}
         (a, b, c) = (ab)c - a(bc), 
    \end{equation}

    \noindent
    is symmetric in $a, b$. That is, 
    
    \begin{equation}
        (a, b, c) = (b, a, c) \iff (ab)c - a(bc) = (ba)c - b(ac). \label{lSID}
    \end{equation}
\end{definition}

\noindent
If $\mathcal{A}$ is an associative algebra, then we have that $(a, b, c) = 0$.  

\begin{remark}
    LSAs are also called \textit{pre-Lie algebras} for reasons seen below.
\end{remark}

\begin{proposition}
    Let $\mathcal{A}$ be a left-symmetric algebra. Then, the commutator

    \begin{equation}
        [a, b]=ab - ba, \quad \forall a, b \in \mathcal{A},
        \label{commutator}
    \end{equation}

    \noindent
    defines a Lie algebra $\mathfrak{g}(\mathcal{A)}$. This Lie algebra is called the \textbf{sub-adjacent Lie algebra of $\mathcal{A}$}; moreover, $\mathcal{A}$ is called a compatible left-symmetric structure on the Lie algebra $\mathfrak{g}(\mathcal{A)}$. 
    \label{prop-subadjacent-Lie-algebra}
\end{proposition}

\begin{proof} We can verify the Lie algebra axioms. The product $[a, b]$ is trivially linear in $a$ and $b$. Moreover, we can see that $[a, a]=0$. Lastly, for all $a, b, c \in \mathcal{A}$, we check the Jacobi identity
\begin{align*}
    [a, [b, c]] &=  a(bc - cb) - (bc - cb)a \\
                &= a(bc) - a(cb) - (bc)a + (cb)a.
\end{align*}

\noindent
We obtain similar expressions for $[c, [a, b]]$ and $[b, [c, a]]$ such that
\begin{align*}
        [a, [b, c]] + [c, [a, b]] + [b, [c, a]] &= a(bc) - a(cb) - (bc)a + (cb)a \\ & \quad + c(ab) - c(ba) - (ab)c + (ba)c \\ & \quad + b(ca) - b(ac) - (ca)b + (ac)b.
\end{align*}

Now, knowing $\mathcal{A}$ is a left-symmetric algebra, using the symmetry identity, we see $(a, b, c) = (b, a, c)$, $(b, c, a) = (c, b, a)$, and $(c, a, b) = (a, c, b)$. Thus

    \[
        [a, [b, c]] + [c, [a, b]] + [b, [c, a]] = 0.
  \] 
  
\end{proof}
\noindent
For more details on left-symmetric algebras, refer to \cite{Chen-li, bai}. 

Proposition \ref{prop-subadjacent-Lie-algebra} implies that every LSA has an underlying Lie algebra structure, hence the name ``Pre-Lie algebras.'' Thus, a Lie algebra $\mathfrak{g}$ is said to \textit{admit} a left-symmetric structure if $\mathfrak{g}$ itself is the sub-adjacent Lie algebra of that LSA. Concretely, $\mathfrak{g}$ admits a left-symmetric structure $\mathcal{A}$ if there exists a bilinear product $(\cdot \mathbin{,} \cdot) \colon \mathfrak{g} \times \mathfrak{g} \to \mathfrak{g}$ satisfying the left-symmetric identity in Eq. \eqref{lSID} to ensure that the failure of associativity is symmetric, and Eq. \eqref{commutator} so that the commutator of the product gives the Lie bracket. See \cite{burde1} for more details.

\begin{definition}[Left-symmetric superalgebras]
    \normalfont
    Let $\mathcal{A} = \mathcal{A}_{\bar0} \oplus \mathcal{A}_{\bar1}$ be a superalgebra. Then, $\mathcal{A}$ is a \textit{left-symmetric superalgebra} (or LSSA, for short) if the associator 
    \begin{equation}
        (a, b, c) = (ab)c - a(bc)
    \end{equation}

    \noindent
    of $\mathcal{A}$ is supersymmetric in both $a$ and $b$, that is

    \begin{equation}
        (a, b, c) = (-1)^{\abs{a} \abs{b}}(b, a, c), \quad \forall a, b, c \in \mathcal{A}.
    \end{equation}
\end{definition}

    \noindent
    Let $\mathcal{A}$ be an LSSA. Then, the supercommutator in Eq. \eqref{supercommutator}

    \[
        [a, b] = ab-(-1)^{\abs{a}\abs{b}}ba, \quad \forall a, b \in \mathcal{A},
    \]

    \noindent
    defines a Lie superalgebra $\mathfrak{g}(\mathcal{A})$ called the \textbf{sub-adjacent Lie superalgebra of $\mathcal{A}$}. Additionally, $\mathcal{A}$ is called the \textit{compatible left-symmetric structure on the Lie superalgebra $\mathfrak{g}(\mathcal{A})$}. 

\begin{remark}
    If $\mathcal{A}$ is an LSSA on a Lie superalgebra $\mathfrak{g}$, then $\mathcal{A}_{\bar0}$ is an LSA on $\mathfrak{g}_{\bar0}$. 
\end{remark}

\subsection{Gröbner Bases: A Brief Overview.} \label{Grobner bases}

We now turn to a brief overview of Gröbner bases and related information relevant for the subsequent superization and computations in Section \ref{Superization LSS}. Gröbner bases are particularly useful for determining if a system of polynomial equations, linear or nonlinear, has a solution. We begin with some preliminary definitions and theorems before introducing Gröbner bases.

Let $\mathbb{K}[x_{1},\ldots,x_{n}]$ be a polynomial ring over a field $\mathbb{K}$ (e.g., the rational numbers $\mathbb{Q}$, the real numbers $\mathbb{R}$, or the complex numbers $\mathbb{C}$). For $f_{1},\ldots,f_{k} \in \mathbb{K}[x_{1},\ldots,x_{n}]$, define $V(f_{1},\ldots,f_{k})$ to be the set of solutions to the system of polynomial equations given by 

\begin{equation}
    f_{1}= \cdots = f_{k} = 0.
    \label{system_poly}
\end{equation}

\noindent
Formally,
\[
    V(f_{1},\ldots,f_{k}) = \{(x_1,\ldots,x_n) \in \mathbb{K}^n \mid f_{i}(x_{1},\ldots,x_{n}) = 0, i = 1, 2,\ldots, k\}.
\]
\noindent
$V(f_{1},\ldots,f_{k})$ is the \textit{algebraic variety} defined by the system of equations $\{f_{1},\ldots,f_{k}\}$.

The system of equations in \eqref{system_poly} determines an ideal defined as follows 
\[
    I = \langle f_{1},\ldots,f_{k} \rangle \subset \mathbb{K}[x_{1},\ldots,x_{n}].
\]

\noindent
It is easy to verify this. Solving this system of equations is equivalent to determining whether this ideal admits a common zero in an appropriate extension of $I$, or what Adams and Loustaunau in \cite{adams-grobner-bases} refer to as a ``better'' generating set for $I$ (this is the Gröbner basis for $I$, as we will see later). That is, any solution to the system solves every polynomial in $I$. The following theorem is particularly relevant to the computations in Section \ref{Superization LSS} for determining if $\mathfrak{osp}(1|2)$ admits a left-symmetric superalgebra structure. 

\begin{theorem}[Hilbert's Basis Theorem]
    In the polynomial ring $\mathbb{K}[x_{1},\ldots,x_{n}]$,  we have the following:
    \begin{enumerate}
    \setlength{\itemsep}{0pt}   % Space between items
    \setlength{\parskip}{0pt}   % Space between paragraphs in items
    \setlength{\topsep}{0pt}    
        \item [(i)] If $I$ is any ideal of $\mathbb{K}[x_{1},\ldots,x_{n}]$, then there exist $f_{1},\ldots,f_{k} \in \mathbb{K}[x_{1},\ldots,x_{n}]$ such that $I = \langle f_{1}, \ldots,f_{k} \rangle$.

        \item [(ii)] If $I_{1} \subseteq I_{2} \subseteq \cdots \subseteq I_{n} \subseteq \cdots$ is an ascending chain of ideals of $\mathbb{K}[x_{1},\ldots,x_{n}]$, then there exists $N$ such that $I_{N} = I_{N+1} = I_{N+2} = \cdots$.
     \end{enumerate}
\end{theorem}

\noindent
\begin{proof}
    See Theorem 1.1.1 (p. 5) in \cite{adams-grobner-bases}.
\end{proof}

By Hilbert's Basis Theorem, every ideal in $\mathbb{K}[x_{1},\ldots,x_{n}]$ is finitely generated. Therefore, even if a system of polynomial equations is infinite, its solution set is determined by finitely many generators of the associated ideal. Now, given an ideal $I \subset \mathbb{K}[x_{1},\ldots,x_{n}]$ and a polynomial $g \in \mathbb{K}[x_{1},\ldots,x_{n}]$, how can we determine if $g \in I$? Namely, if $1 \in I$, then there exist polynomials $a_{1},\ldots,a_{k} \in \mathbb{K}[x_{1},\ldots,x_{n}]$ such that 
\begin{equation}
    \sum\limits_{i=1}^{k}a_{i}f_{i} = 1.
\end{equation}\label{identity-poly}

\noindent
Note that the set of polynomials $\{a_{i}\}_{i=1}^{k}$ is called the \textit{Nullstellensatz certificate} and it is dependent on the polynomials defining the system of equations. Furthermore, Nullstellensatz certificates are not necessarily unique. See \cite{bart-nullstellensatz} for more details on Nullstellensatz certificates. 

Now, let ${\bf{x}} = (x_{1},\ldots,x_{n}) \in \mathbb{K}^{n}$ be a solution to the system. Then, we have
\[
    0 = \sum\limits_{i=1}^{k}a_{i}({\bf{x}})f_{i}({\bf{x}}) = 1,
\]

\noindent
a contradiction. Thus, $1 \in I$ implies that the system of equations does not have a solution. This is precisely outlined in the following theorem. Recall that a field $\mathbb{K}$ is \textit{algebraically closed} if for every nonconstant polynomial $f \in \mathbb{K}[x_{1},\ldots,x_{n}]$ the equation $f = 0$ has a solution in $\mathbb{K}^{n}$ \cite{judson-algebra}.

\begin{theorem}[Hilbert's Weak Nullstellensatz] Let $\mathbb{K}$ be an algebraically closed field and let $I \subset \mathbb{K}[x_{1},\ldots,x_{n}]$ be an ideal. Then,
$V(I) = \varnothing$ if and only if $I = \langle 1 \rangle$. \label{nullstellensatz}
\end{theorem}

\noindent
Recalling that $\langle 1 \rangle$ is the \textit{principal unit ideal} \cite{judson-algebra}, this theorem is equivalent to saying $I$ has no common solutions if and only if $I = \mathbb{K}[x_{1},\ldots,x_{n}]$. 

On the problem of solving systems of polynomial equations, particularly nonlinear ones, Theorem \ref{nullstellensatz} proves to be fundamental. This is especially true as it serves as the theoretical justification for Gröbner bases. However, before defining Gröbner bases, we present the concepts of term orders and polynomial reduction. 

To study such ideals and their solutions using Gröbner bases, we must first establish a \textit{term order} $\succ$ on $\mathbb{K}[x_{1},\ldots,x_{n}]$. That is, we establish a total well-ordering on monomials that is compatible with multiplication. Let ${\mathbf{x}}^{\boldsymbol{\alpha}} = x_{1}^{\alpha_{1}} \cdots x_{n}^{\alpha_{n}}$ and ${\mathbf{x}}^{\boldsymbol{\beta}} = x_{1}^{\beta_{1}} \cdots x_{n}^{\beta_{n}}$ be monomials such that $x_{1} \succ x_{2} \succ \cdots \succ x_{n-1} \succ x_{n}$ and
\[
    \boldsymbol{\alpha} = (\alpha_{1},\ldots,\alpha_{n}), \boldsymbol{\beta} = (\beta_{1},\ldots, \beta_{n}) \in \mathbb{Z}_{>0}^{n}.
\] 
\noindent
Moreover, let $|\boldsymbol{\alpha}|$ denote the sum of all the components of $\boldsymbol{\alpha}$. 
There are three important term orders used in the Gröbner algorithm introduced in \cite{buchberger-grobner}: 
    \begin{enumerate}
    \setlength{\itemsep}{0pt}   % Space between items
    \setlength{\parskip}{0pt}   % Space between paragraphs in items
    \setlength{\topsep}{0pt} 
        \item [(1)] Lexicographical Order (lex): ${\mathbf{x}}^{\boldsymbol{\alpha}} \succ_{\mathrm{lex}} {\mathbf{x}}^{\boldsymbol{\beta}}$ if the leftmost nonzero entry of $\boldsymbol{\alpha} - \boldsymbol{\beta}$ is positive. Note that, for this order, $x_{1}^{\mu} \succ_{\mathrm{lex}} x_{2}^{\nu}$ holds for all $\mu, \nu \in \mathbb{Z}^{+}$.
        \item [(2)] Degree Reverse Lexicographical Order (degrevlex): 
        \[
           {\mathbf{x}}^{\boldsymbol{\alpha}} \succ_{\mathrm{degrevlex}} {\mathbf{x}}^{\boldsymbol{\beta}} \iff \begin{cases}
               |\boldsymbol{\alpha}| > |\boldsymbol{\beta}|, \text{ or} \\
               |\boldsymbol{\alpha}| = |\boldsymbol{\beta}| \text{ and the rightmost nonzero entry of} \\[-4pt] \ \  \boldsymbol{\alpha} - \boldsymbol{\beta} \text{ is negative.}
           \end{cases}
        \]
        \item [(3)] Degree Lexicographical Order (deglex): 
        \[
            {\mathbf{x}}^{\boldsymbol{\alpha}} \succ_{\mathrm{deglex}} {\mathbf{x}}^{\boldsymbol{\beta}} \iff \begin{cases}
                |\boldsymbol{\alpha}| > |\boldsymbol{\beta}|, \text{ or} \\ |\boldsymbol{\alpha}| = |\boldsymbol{\beta}| \text{ and the leftmost nonzero entry of} \\[-4pt] \ \ \boldsymbol{\alpha} - \boldsymbol{\beta} \text{ is positive}.
            \end{cases}
        \]
    \end{enumerate}
The choice of term order does affect the result of polynomial reduction.  For more details on term orders, see Chapter 1.4 in \cite{adams-grobner-bases}.

\begin{example}\label{example1-term-order}
    \normalfont
    Fix a lex term order. Consider the polynomial ring $\mathbb{K}[x_1, x_2, x_3]$ and the monomials $x_{1}^2x_{2}$ and $x_{1}x_{2}^{3}x_{3}^{2}$. We see that $x_{1}^2x_{2} \succ x_{1}x_{2}^{3}x_{3}^{2}$ as 
    \[
      \boldsymbol{\alpha} - \boldsymbol{\beta} = (2,1,0) - (1,3,2) = (1,-2,-2).  
    \]
    \noindent
    The leftmost nonzero entry $(\boldsymbol{\alpha} - \boldsymbol{\beta})_{1} = 1$ is clearly positive.
    
    In the polynomial ring $\mathbb{K}[x_{1}, x_{2}]$, we have 
    \[
        \cdots \succ x_{1}^{2} \succ \cdots \succ x_{1}x_{2}^{2} \succ x_{1}x_{2} \succ x_{1} \succ \cdots \succ x_{2}^{3} \succ x_{2}^{2} \succ x_{2} \succ 1.
    \]
\end{example}

\begin{example}
    \normalfont
    Fix a degrevlex term order. Consider the polynomial ring $\mathbb{K}[x_1, x_2, x_3]$. The monomials $x_{1}^{2}x_{2}x_{3}$ and $x_{1}x_{2}^{3}$ are such that $x_{1}^{2}x_{2}x_{3} \prec x_{1}x_{2}^{3}$ since 
    \[
       \boldsymbol{\alpha} - \boldsymbol{\beta} = (2,1,1) - (1,3,0) = (1,-2, 1). 
    \]
    \noindent
    In this case, the rightmost nonzero entry $(\boldsymbol{\alpha} - \boldsymbol{\beta})_{3} = 1$ is positive. 
\end{example}

\begin{example}
    \normalfont
    Fix a deglex term order. Consider the same polynomial ring and monomials as in Example \ref{example1-term-order}. We see that 
    \[
        |\boldsymbol{\alpha}| - |\boldsymbol{\beta}| = 3 - 6 = -3.
    \]
    \noindent
    Therefore, $x_{1}^2x_{2} \prec x_{1}x_{2}^{3}x_{3}^{2}$.

    In the two-variable polynomial ring $\mathbb{K}[x_{1}, x_{2}]$, we have 
    \[
    \cdots \succ x_{1}^{3} \succ x_{1}^{2}x_{2} \succ x_{1}x_{2}^2 \succ x_{2}^{3} \succ x_{1}^{2} \succ x_{1}x_{2} \succ x_{2}^{2} \succ x_{1} \succ x_{2} \succ 1.
    \]
\end{example}

Now, the concept of polynomial \textit{reduction} is described by the Euclidean Algorithm and substantiated by the Division Algorithm. Recalling some concepts in one variable, given a polynomial $f \in \mathbb{K}[x]$ such that $f \neq 0$, the \textit{degree} of $f$, denoted $\mathrm{deg}(f)$, is the largest exponent of $x$ appearing in $f$. The \textit{leading term} of $f$, denoted $\text{lt}(f)$, is the monomial of $f$ with the highest degree. The \textit{leading coefficient} of $f$, denoted $\mathrm{lc}(f)$, is the coefficient of the leading term. The \textit{leading power product} of $f$, denoted $\mathrm{lp}(f)$, is the leading monomial of $f$ without its coefficient. Consider the following example for more clarity. 

\begin{example}
    \normalfont
    For the case of one variable, consider the general form of a nonzero polynomial 
    \[f(x) = \sum\limits_{k=1}^{n} a_{k}x^{k} = a_{0} + a_{1}x + \cdots + a_{n-1}x^{n-1} + a_{n}x^{n} \in \mathbb{K}[x],\]
    \noindent
    where $a_{k} \in \mathbb{K}$ and $a_{n} \neq 0$. Trivially, we have that $\mathrm{deg}(f) = n \in \mathbb{Z}^{+}$, $\mathrm{lt}(f) = a_{n}x^{n}$, $\mathrm{lc}(f) = a_{n}$, and $\mathrm{lp}(f) = x^{n}$.
\end{example}

\begin{example}
    \normalfont
    For the multivariate case, these notions differ. We must first fix a term order before defining these concepts. Let $\succ$ be a term order. Consider a nonzero multivariate polynomial given by 
    \[
        g({\mathbf{x}}) = \sum\limits_{\boldsymbol{\beta}} a_{\boldsymbol{\beta}}{\mathbf{x}}^{\boldsymbol{\beta}} \in \mathbb{K}[x_{1},\ldots,x_{n}], 
    \]
    \noindent
    where ${\mathbf{x}}^{\boldsymbol{\beta}} = x_{1}^{\beta_{1}} \cdots x_{n}^{\beta_{n}}$, and $\boldsymbol{\beta} = (\beta_{1},\ldots, \beta_{n}) \in \mathbb{Z}_{>0}^{n}$. The \textit{total degree} of a monomial $\mathbf{x}^{\boldsymbol{\beta}} = x_{1}^{\beta_{1}} \ldots x_{n}^{\beta_{n}}$ is given by
    \[
        \mathrm{deg}(\mathbf{x}^{\boldsymbol{\beta}}) = \sum\limits_{i=1}^{n} \beta_{i} = |\boldsymbol{\beta}|.
    \]
    \noindent
    Hence, the degree of $g$ is defined by
    \[
        \mathrm{deg}(g) = \mathrm{max}\{|\boldsymbol{\beta}| \mid a_{\boldsymbol{\beta}} \neq 0\}.
    \]
\end{example}

\begin{example}
    \normalfont
    Consider the polynomial $f = 3x_{1}^{2}x_{2}x_{3}^4 + 2x_{1}x_{2}^{3} - x_{2}^{4}x_{3} \in \mathbb{K}[x_1, x_2, x_3].$ Fix a lex term order $x_1 \succ x_2 \succ x_3$. We see that 
    \[
        \mathrm{deg}(f) = 7, \ \mathrm{lt}(f) = 3x_{1}^{2}x_{2}x_{3}^4, \ \mathrm{lc}(f) = 3, \ \mathrm{lp}(f) = x_{1}^{2}x_{2}x_{3}^4. 
    \]
\end{example}

Consider two polynomials $f,g \in \mathbb{K}[x]$. Using the Division Algorithm, we can iteratively reduce $f$ modulo $g$ to obtain another polynomial $h \in \mathbb{K}[x]$. This process is denoted by
\[
    f \xrightarrow{\hspace{0.2cm} g \hspace{0.2cm}} h.
\]
\noindent
In reducing $f$ modulo $g$, a suitable monomial multiple of $g$ is subtracted from $f$, resulting in the polynomial $h$. For more details on polynomial reduction, see \cite{buchberger-intro-grobner, adams-grobner-bases}. That is, we can reduce $f$ modulo $g$ repeatedly and denote it as follows
\[
    f \xrightarrow{\hspace{0.2cm} g \hspace{0.2cm}}_{+} r,
\]

\noindent
where $r$ is the remainder polynomial with $\mathrm{deg}(r) < \mathrm{deg}(f)$. 

We can generalize these concepts to the multivariate case, which is useful in understanding how Gröbner bases work. We fix a term order on $\mathbb{K}[x_{1},\ldots, x_{n}]$.

\begin{definition}
    \normalfont
    Let $f, h$ and $g_{1},\ldots,g_{k}$ be polynomials in $\mathbb{K}[x_{1},\ldots,x_{n}]$ with $g_{i} \neq 0$ for $i=1,\ldots,k$. Let $G = \{g_{1},\ldots,g_{k}\}$. We say $f$ \textit{reduces} to $h$ modulo $G$, denoted 
    \[
        f \xrightarrow{\hspace{0.2cm} G \hspace{0.2cm}}_{+} h
    \]

    \noindent
    if and only if there exists a sequence of indices $\{i_s\}_{s=1}^{t} \subset \{1,\ldots,k\}$ and a sequence of polynomials $\{h_{s}\}_{s=1}^{t-1} \subset \mathbb{K}[x_{1},\ldots,x_{n}]$ such that
    \begin{equation}
        f \xrightarrow{\hspace{0.2cm} g_{i_{1}} \hspace{0.2cm}} h_{1} \xrightarrow{\hspace{0.2cm} g_{i_{2}} \hspace{0.2cm}} h_{2} \xrightarrow{\hspace{0.2cm} g_{i_{3}} \hspace{0.2cm}} \cdots \xrightarrow{\hspace{0.2cm} g_{i_{t-1}} \hspace{0.2cm}} h_{t-1} \xrightarrow{\hspace{0.2cm} g_{i_{t}} \hspace{0.2cm}} h. \label{poly_red}
    \end{equation}
\end{definition}
\noindent
The terminal polynomial $h$ of this reduction is called the \textit{normal form} of $f$ modulo $G$. This definition implies that a polynomial $f$ can be reduced modulo $G$ finitely many times, as seen above. In general, many, but not infinitely many, reductions are possible. If $h$ cannot be reduced further with respect to $G$, we write $\underline{h}_{G}$.

Now we define the concept of a \textit{reduced} polynomial, equivalent to the remainder polynomial.

\begin{definition}
    \normalfont
    A polynomial $r$ is called \textit{reduced} with respect to a set of non-zero polynomials $G = \{g_{1},\ldots,g_{k}\}$ if $r=0$ or no power product that appears in $r$ is divisible by any one of the $\mathrm{lp}(g_{i})$, for $i=1,\ldots,k$.  That is, $r$ cannot be reduced modulo $G$.
\end{definition}

\noindent
If $f \xrightarrow{\hspace{0.2cm} G \hspace{0.2cm}}_{+} r$ and $r$ is reduced with respect to $F$, then $r$ is called a remainder for $f$ with respect to $G$. 

Multivariate polynomial reduction always terminates but is not necessarily unique, as seen in Eq.  \eqref{poly_red}. That is, given $f$ and $G$, there may exist $h_{1}, h_{2}$ such that either 
\[f \xrightarrow{\hspace{0.2cm} G \hspace{0.2cm}}_{+} h_{1}, \ \text{ or } \ f \xrightarrow{\hspace{0.2cm} G \hspace{0.2cm}}_{+} h_{2}.\]
For more information on termination of multivariate polynomial reduction, see \cite{buchberger-grobner}.

\begin{theorem}
    \label{theorem-div-algo}
    Given a set of nonzero polynomials $G = \{g_{1}, \ldots, g_{k}\}$, and a polynomial $f \in \mathbb{K}[x_{1}, \ldots, x_{n}],$ the Division Algorithm produces polynomials $u_{1},\ldots, u_{k}, r \in \mathbb{K}[x_{1},\ldots,x_{n}]$ such that
    \[
        f = u_{1}g_{1} + \cdots + u_{k}g_{k} + r,
    \]
    \noindent
    with $r$ reduced with respect to $G$ and 
    \[
    \mathrm{lp}(f) = \max(\max\limits_{i \in [1,k]} (\mathrm{lp}(u_{i}) \mathrm{lp}(g_{i})), \mathrm{lp}(r)).
    \]
\end{theorem}

\begin{proof}
 See Theorem 1.5.9 (p. 30) in \cite{adams-grobner-bases}.
\end{proof}

Having defined multivariate polynomial reduction, we can now define the fundamental object used later on in our exploration of left-symmetric structures on $\mathfrak{osp}(1|2)$: Gröbner bases. 

\begin{definition}[Gröbner Basis]
    \normalfont
    Let $G = \{g_{1},\ldots,g_{k}\}$ be a set of nonzero polynomials contained in an ideal $I$. G is called a \textit{Gröbner basis} if and only if for all $f \in I$, with $f \neq 0$, there exists an $i \in \{1,\ldots,k\}$ such that $\mathrm{lp}(g_{i}) \mid \mathrm{lp}(f).$
\end{definition}

We present three additional characterizations of a Gröbner basis. However, we first consider the following definition. Given a subset $S \subset \mathbb{K}[x_{1},\ldots,x_{n}]$, define the \textit{leading term ideal} of $S$ by
\[
    \mathrm{Lt}(S) = \langle \mathrm{lt}(s) \mid s \in S \rangle.
\]
\begin{proposition}
    Let $I$ be a nonzero ideal of $\mathbb{K}[x_{1},\ldots,x_{n}]$. The following statements are equivalent for a set of nonzero polynomials $G = \{g_{1},\ldots,g_{k}\} \subseteq I$.
    \begin{enumerate}
     \setlength{\itemsep}{0pt}   % Space between items
    \setlength{\parskip}{0pt}   % Space between paragraphs in items
    \setlength{\topsep}{0pt}
    \item [(i)] $G$ is a Gröbner basis for I.
    \item [(ii)] $f \in I$ if and only if $f \xrightarrow{\hspace{0.2cm} G \hspace{0.2cm}}_{+} 0$. 
    \item [(iii)] $f \in I$ if and only if $f = \sum\limits_{i=1}^{k} h_{i}g_{i}$ with $\mathrm{lp}(f) = \max\limits_{i \in [1,k]} (\mathrm{lp}(h_{i})\mathrm{lp}(g_{i}))$.
    \item [(iv)] $\mathrm{Lt}(G) = \mathrm{Lt}(I)$.
    \end{enumerate}
\end{proposition}
    
\begin{proof}
    See Theorem 1.6.2 (p. 32) in \cite{adams-grobner-bases}. 
\end{proof}

\noindent
In other words, if $G$ is a Gröbner basis for $I$, then there are no nonzero polynomials in $I$ that are reduced modulo $G$. This is mainly because, by Theorem \ref{theorem-div-algo}, if $f \in \mathbb{K}[x_{1},\ldots, x_{n}]$, then there exists an $r \in \mathbb{K}[x_{1},\ldots, x_{n}]$ reduced with respect to $G$ with $f \xrightarrow{\hspace{0.2cm} G \hspace{0.2cm}}_{+} r$. Hence, $f-r \in I$ and $f \in I$ if and only if $r \in I$. If $r = 0$, then trivially $f \in I$. Conversely, if $f \in I$ and $r \neq 0$, then $r \in I$, and by (i), there exists $i \in \{1,\ldots,k\}$ such that $\mathrm{lp}(g_{i})$ divides $\mathrm{lp}(r)$. Yet, $r$ is reduced with respect to $G$. Thus, $r=0$. 

\begin{corollary}
    If $G = \{g_{1}, \ldots, g_{k}\}$ is a Gröbner basis for the ideal $I$, then we have that $I = \langle g_{1}, \ldots, g_{k} \rangle$. 
\end{corollary}

\begin{proof}
    See Corollary 1.6.3 (p. 33) in \cite{adams-grobner-bases}.
\end{proof}

The previous definition and proposition, along with other consequences found in \cite{adams-grobner-bases}, lead to the following corollary and theorem.

\begin{corollary}
    Every nonzero ideal $I$ of $\mathbb{K}[x_{1},\ldots,x_{n}]$ has a Gröbner basis.
\end{corollary}

\begin{proof}
    See Corollary 1.6.5 (p. 34) in \cite{adams-grobner-bases}.
\end{proof}

\begin{theorem}
    Let $G = \{g_{1},\ldots,g_{k}\}$ be a set of nonzero polynomials in $\mathbb{K}[x_{1},\ldots,x_{n}]$. Then, $G$ is a Gröbner basis if and only if for all $f \in \mathbb{K}[x_{1},\ldots,x_{n}]$ the remainder of the reduction of $f$ modulo $G$ is unique. \label{theorem_grobner2}
\end{theorem}

\begin{proof}
    See Theorem 1.6.7 (p. 34) in \cite{adams-grobner-bases}.
\end{proof}

Theorem \ref{theorem_grobner2}, equivalently, states that a Gröbner basis is a generating set $G$ of an ideal such that every normal form of a polynomial $f \in I$ equals zero modulo $G$. A fundamental consequence of this theorem, as seen earlier with Theorem \ref{nullstellensatz}, is that an ideal $I$ contains the constant polynomial $1$ if and only if its Gröbner basis $G$ contains $1$. 

\section{Left-Symmetric Structures on \texorpdfstring{$\mathfrak{sl}_{2}(\mathbb{K})$}{sl2(k)}} \label{LSS on sl2(k)}

The study of LSAs on Lie algebras and superalgebras has been carried out extensively by several experts (see \cite{auslander, kim, perea}). The study of left-symmetric algebra structures on Lie algebras has been extensive. However, a complete existence and classification theory over arbitrary fields, especially in positive characteristic, is still not available. Helmstetter showed in \cite{helmstetter} that every Lie algebra $\mathfrak{g}$ over a field $\mathbb{K}$ of characteristic $0$ such that $[\mathfrak{g}, \mathfrak{g}]=\mathfrak{g}$ does not admit LSAs. Additionally, if $\mathfrak{g}$ is semisimple over a field $\mathbb{K}$ of characteristic $0$, then $\mathfrak{g}$ does not admit LSAs either. 

In \cite{burde1}, Burde studies left-symmetric structures on simple modular Lie algebras. In particular, he proves that if $\mathfrak{g}$ is a classical simple Lie algebra over a field of characteristic $p > 3$ and $\mathfrak{g}$ admits a compatible left-symmetric algebra structure, then $p$ divides $\mathrm{dim}(\mathfrak{g})$. Recall that an LSA is a non-associative algebra $\mathcal{A}$ such that the associator $(a,b,c)$ is symmetric in $a$ and $b$
\[
(a, b, c) = (b, a, c) \quad \forall a, b, c \in \mathcal{A}.
\]
Moreover, every LSA gives rise to a Lie algebra through the commutator bracket $[x, y] = xy - yx$, and a left-symmetric structure on a Lie algebra $\mathfrak{g}$ is a left-symmetric algebra structure on the vector space of $\mathfrak{g}$ whose commutator reproduces the Lie bracket of $\mathfrak{g}$.

For the present project, the relevant case is $\mathfrak{g} = \mathfrak{sl}_{2}(\mathbb{K})$, since the even part of $\mathfrak{osp}(1|2)$ is isomorphic to $\mathfrak{sl}_{2}(\mathbb{K})$. Since $\mathrm{dim}(\mathfrak{sl}_{2}(\mathbb{K})) = 3$, Burde's result rules out compatible left-symmetric algebra structures on $\mathfrak{sl}_{2}(\mathbb{K})$ for $p > 3$. The remaining exceptional case is $p = 3$, where Burde explicitly classifies the compatible structures. Burde’s result for the Lie algebra $\mathfrak{sl}_2(\mathbb{K})$ is described below. 

\begin{theorem}[Burde, Specialized to $\mathfrak{sl}_{2}(\mathbb{K})$] \label{burde-theorem-specialized}
    Let $\mathbb{K}$ be an algebraically closed field of characteristic $p>2$, $\mathfrak{g} \coloneqq \mathfrak{sl}_2(\mathbb{K})$ and consider its standard basis $\langle E, F, H \rangle$.
    Then, $\mathfrak{sl}_2(\mathbb{K})$ admits left-symmetric structures if and only if $p = 3$. In characteristic $3$, these structures are classified into four families, denoted $\mathcal{A}_{1,u,v}$, $\mathcal{A}_{2, \gamma}$, $\mathcal{A}_{3}$, and $\mathcal{A}_{4}$.
\end{theorem}

\begin{proof}
    If $p=3$, then $\mathfrak{g}$ admits left-symmetric structures. These are the algebras defined by the matrices $\lambda_{E}, \lambda_{F}, \lambda_{H}$ seen later in Eqs. \eqref{lss1}, \eqref{lss2}, \eqref{lss3}, and \eqref{lss4}. 
    
    Now, assume that $\mathfrak{g}$ admits a left-symmetric structure $M_{\lambda}$ and that $p > 3$. Any left-symmetric structure on $\mathfrak{g}$ induces a Lie algebra representation
    \[
        \lambda \colon  \mathfrak{g} \to \mathrm{End}(\mathfrak{g}), \quad x \mapsto \lambda_{x},
    \]
    \noindent
    where $\lambda_{x}(y) \coloneqq x*y$ is the left multiplication by $x$. The vector space $\mathfrak{g}$, equipped with the action $x*v \coloneqq \lambda_x(v)$, therefore becomes a $\mathfrak{g}$-module of dimension $3$, which we denote by $M_{\lambda}$. More information on left-symmetric $\mathfrak{g}$-modules can be found in \cite{burde1}.
    
    Since every left-symmetric structure over $\mathbb{K}$ can be defined over the algebraic closure of $\mathbb{K}$, it is enough to prove the following:
    \begin{enumerate}
        \setlength{\itemsep}{0pt}   % Space between items
        \setlength{\parskip}{0pt}   % Space between paragraphs in items
        \setlength{\topsep}{0pt}
        \item [(1)] Let $\mathbb{K}$ be an algebraically closed field of characteristic $p > 3$ and $M$ be any $3$-dimensional $\mathfrak{g}$-module. Then, $H^{1}(\mathfrak{g}, M) = 0$.
    \end{enumerate}

    \noindent
    To establish (1), consider the following theorem by Dzhumadil'daev \cite{dzhumadildaev-cohomology} for $\mathfrak{g} \coloneqq \mathfrak{sl}_{2}(\mathbb{K})$ (Theorem 4): if $\mathbb{K}$ is an algebraically closed field and $E$ is an irreducible $\mathfrak{g}$-module, then
    \begin{enumerate}
        \setlength{\itemsep}{0pt}   % Space between items
        \setlength{\parskip}{0pt}   % Space between paragraphs in items
        \setlength{\topsep}{0pt}
        \item [(2)] $H^{1}(\mathfrak{g}, E) \cong \begin{cases}
            \mathbb{K} \oplus \mathbb{K}, &\text{if dim}(E) = p -1,\\ 0, & \text{otherwise.}
        \end{cases}$
    \end{enumerate}
    Consider a composition series for $M$. All irreducible composition factors have dimension at most $3$. Given that $p-1 > 3$, none of these  composition factors have dimension $p-1$, thus all factors have trivial $1$-cohomology by (2). Using the long exact sequence then yields $H^{1}(\mathfrak{g}, M) = 0$, establishing (1).

    Consider now the following proposition (Proposition 1.2.8) found in \cite{burde1}
    
    \begin{proposition} \label{prop-burde}
        Let $\mathfrak{g}$ be a Lie algebra over $\mathbb{K}$ such that $[\mathfrak{g}, \mathfrak{g}] = \mathfrak{g}$ and assume that $\mathfrak{g}$ admits a left-symmetric structure $M_{\lambda}$. Then, $H^{1}(\mathfrak{g}, M_{\lambda}) = 0$ implies $p > 0$ and $p \mid \mathrm{dim}(\mathfrak{g})$.
    \end{proposition}
    
    \noindent
    Applying (1) to the specialized module $M = M_{\lambda}$ gives $H^{1}(\mathfrak{g}, M_{\lambda}) = 0$. Knowing $[\mathfrak{g}, \mathfrak{g}] = \mathfrak{g}$, where $\mathfrak{g} = \mathfrak{sl}_{2}(\mathbb{K})$, Proposition \ref{prop-burde}, forces $p \mid \mathrm{dim}(\mathfrak{g}) = 3$, contradicting $p > 3$. 
\end{proof}

The proof relies on a cohomological analysis associated with $\mathfrak{sl}_2(\mathbb{K})$ in characteristic $3$. Burde constructs explicit examples of such structures in this modular setting, demonstrating the contrast with the characteristic $0$ case, where $\mathfrak{sl}_2(\mathbb{K})$ admits no left-symmetric algebra structures. Given the complexity of the proof, we will not focus on its details. See \cite{burde1} and Chapter 4 of \cite{burde-thesis} for more details.

If $\mathbb{K}$ is an algebraically closed field of characteristic 3, it is possible to classify all left-symmetric structures on $\mathfrak{sl}_2(\mathbb{K})$, as seen in Theorem \ref{burde-theorem-specialized}. Letting $(\mathcal{A}, \lambda)$ be a left-symmetric algebra with associated Lie algebra $\mathfrak{sl}_2(\mathbb{K})$, $(\mathcal{A}, \lambda)$ is isomorphic to one of the following algebras: (i) $\mathcal{A}_{1,u,v}$, (ii) $\mathcal{A}_{2,\gamma}$, (iii) $\mathcal{A}_{3}$, or (iv) $\mathcal{A}_{4}$. These structures are defined by the matrices $\lambda_{E}, \lambda_{F}, \text{ and }\lambda_{H}$ given below.

% \begin{equation}
%     \begin{pmatrix} 0 & u & 1 \\ 0 & w & -1 \\ 1 & w^2 + 1 & -w\end{pmatrix},  \begin{pmatrix} u & u - vw^2 & vw \\ w & v & 1 \\ w^2 & -vw & -u -v\end{pmatrix}, \begin{pmatrix} 0 & vw & -v \\ -1 & -1 & 0 \\ -w & -u -v & 1 \end{pmatrix},
%     \label{lss1}
% \end{equation}

% \begin{equation}
%     \begin{pmatrix} 0 & 0 & \gamma \\ 0 & 0 & 0 \\ 0 & -1 -\gamma^{-1} & 0
%     \end{pmatrix}, \begin{pmatrix} 0 & 0 & 0 \\ 0 & 0 & \gamma \\ 1 -\gamma^{-1} & 0 & 0 \end{pmatrix}, \begin{pmatrix} \gamma - 1 & 0 & 0 \\ 0 & \gamma + 1 & 0 \\ 0 & 0 & \gamma \end{pmatrix},
%     \label{lss2}
% \end{equation}

% \begin{equation}
%     \begin{pmatrix}
%         0 & 0 & -1 \\
%         0 & 0 & 0 \\
%         0 & 0 & 0
%     \end{pmatrix},
%     \begin{pmatrix}
%         0 & 1 & 1 \\
%         0 & 0 & -1 \\
%         -1 & -1 & 0
%     \end{pmatrix},
%     \begin{pmatrix}
%         1 & 1 & 0 \\
%         0 & 0 & 0 \\
%         0 & 0 & -1
%     \end{pmatrix}
%     \label{lss3}
% \end{equation}

% \begin{equation}
%     \begin{pmatrix}
%         0 & 0 & \epsilon \\
%         0 & 0 & 0 \\
%         0 & \epsilon -1 & 0
%     \end{pmatrix},
%     \begin{pmatrix}
%         0 & 1 & 0 \\
%         0 & 0 & \epsilon \\
%         \epsilon + 1 & 0 & 0
%     \end{pmatrix}, 
%     \begin{pmatrix}
%         \epsilon - 1 & 0 & 0 \\
%         0 & \epsilon + 1 & 0 \\
%         0 & 0 & \epsilon
%     \end{pmatrix},
%     \label{lss4}
% \end{equation}

\begin{alignat}{3}
    &\begin{pmatrix} 0 & u & 1 \\ 0 & w & -1 \\ 1 & w^2 + 1 & -w\end{pmatrix} &
    \quad &\begin{pmatrix} u & u - vw^2 & vw \\ w & v & 1 \\ w^2 & -vw & -u -v\end{pmatrix} &
    \quad &\begin{pmatrix} 0 & vw & -v \\ -1 & -1 & 0 \\ -w & -u -v & 1 \end{pmatrix},
    \label{lss1} \\[10pt]
    &\begin{pmatrix} 0 & 0 & \gamma \\ 0 & 0 & 0 \\ 0 & -1 -\gamma^{-1} & 0 \end{pmatrix} &
    \quad &\begin{pmatrix} 0 & 0 & 0 \\ 0 & 0 & \gamma \\ 1 -\gamma^{-1} & 0 & 0 \end{pmatrix} &
    \quad &\begin{pmatrix} \gamma - 1 & 0 & 0 \\ 0 & \gamma + 1 & 0 \\ 0 & 0 & \gamma \end{pmatrix},
    \label{lss2} \\[10pt]
    &\begin{pmatrix} 0 & 0 & -1 \\ 0 & 0 & 0 \\ 0 & 0 & 0 \end{pmatrix} &
    \quad &\begin{pmatrix} 0 & 1 & 1 \\ 0 & 0 & -1 \\ -1 & -1 & 0 \end{pmatrix} &
    \quad &\begin{pmatrix} 1 & 1 & 0 \\ 0 & 0 & 0 \\ 0 & 0 & -1 \end{pmatrix},
    \label{lss3} \\[10pt]
    &\begin{pmatrix} 0 & 0 & \xi \\ 0 & 0 & 0 \\ 0 & \xi -1 & 0 \end{pmatrix} &
    \quad &\begin{pmatrix} 0 & 1 & 0 \\ 0 & 0 & \xi \\ \xi + 1 & 0 & 0 \end{pmatrix} &
    \quad &\begin{pmatrix} \xi - 1 & 0 & 0 \\ 0 & \xi + 1 & 0 \\ 0 & 0 & \xi \end{pmatrix},
    \label{lss4}
\end{alignat}

\noindent
such that $w \coloneqq v - u$, $u, v \in \mathbb{K}$, and $\gamma \in \mathbb{K}^\times, \gamma^{2} \neq -1$ and $\xi^{2} = -1$. The matrices in Eq. \eqref{lss1} correspond to structure $\mathcal{A}_{1,u,v}$, those in Eq. \eqref{lss2} to structure $\mathcal{A}_{2,\gamma}$, and so forth. In Section \ref{Verif Burde LSAs}, we will explicitly represent these structures.

\begin{remark}
    Note that for structure $\mathcal{A}_4$, Burde \cite{burde1} uses the parameter $\beta$ instead of $\xi$. To avoid confusion with the $\beta$-parameters introduced in Section \ref{Computation}, we use $\xi$ instead for the relation $\xi^{2} = -1$.
\end{remark}

\subsection{Left-Symmetric Structures on \texorpdfstring{$\mathfrak{osp}(1|2)$}{osp(1|2)}.} \label{Existence LSS on osp(1|2)}

This capstone project aims to investigate the existence of left-symmetric structures on the Lie superalgebra $\mathfrak{osp}(1|2)$. Since $\mathfrak{osp}(1|2)_{\bar{0}} \cong \mathfrak{sl}_2(\mathbb{K})$, any compatible left-symmetric superalgebra structure on $\mathfrak{osp}(1|2)$ would restrict to a compatible left-symmetric algebra structure on its even part. Thus, Burde's classification in characteristic $3$ provides the possible even-part structures that one can attempt to extend to the full Lie superalgebra $\mathfrak{osp}(1|2)$. The main question of this project is whether any of these four families can be extended to the odd part of $\mathfrak{osp}(1|2)$ in a way that satisfies the graded left-symmetric identity.

The questions I aim to address are:
\begin{enumerate}
    \setlength{\itemsep}{0pt}   % Space between items
    \setlength{\parskip}{0pt}   % Space between paragraphs in items
    \setlength{\topsep}{0pt}
    \item [(i)] Does $\mathfrak{osp}(1|2)$ admit a left-symmetric superalgebra?
    \item [(ii)] Can Burde's approach for $\mathfrak{sl}_2(\mathbb{K})$ in \cite{burde1} be modified or extended to $\mathfrak{osp}(1|2)$?
    \item [(iii)] What constraints do the super-Jacobi identity and parity considerations impose on potential left-symmetric structures on $\mathfrak{osp}(1|2)$?
\end{enumerate}

In developing this project, I plan to check if Burde's left-symmetric algebras are compatible with the odd part of the Lie superalgebra $\mathfrak{osp}(1|2)$. The earlier isomorphism $\mathfrak{osp}(1|2)_{\bar0} \cong \mathfrak{sl}_2(\mathbb{K})$ suggests that the structure of $\mathfrak{osp}(1|2)$ may admit richer left-symmetric superalgebras. 

\subsection{Verification of Burde's Left-Symmetric Algebras.} \label{Verif Burde LSAs}
Note that Burde's matrices $\lambda_{E},$ $ \ \lambda_{F},$ and $\lambda_{H}$ are all in $\mathrm{End}(\mathfrak{sl}_{2}(\mathbb{K}))$. Specifically, they are left-multiplication operators of the Lie algebra $\mathfrak{sl}_{2}(\mathbb{K})$, defined by the $\mathfrak{sl}_{2}(\mathbb{K})$-representation
\[
   \lambda \colon \mathfrak{sl}_{2}(\mathbb{K}) \to \mathrm{End}(\mathfrak{sl}_{2}(\mathbb{K})), \ X \mapsto \lambda_{X}, \ \forall X \in \mathfrak{sl}_{2}(\mathbb{K}).
\]

% left-symmetric product is equivalent to g-module multiplication
\noindent
We can rewrite these explicitly in terms of the $\langle E,F, H\rangle$ basis. Burde encodes the left multiplication $X * Y$ through the $\mathfrak{sl}_{2}(\mathbb{K})$-module $\lambda_{X}$ (the left operator). This means each matrix $\lambda_{X}$ acts on the coordinate vector of $Y$. However, first we must express the vector $Y$ as a coordinate vector in terms of the basis $\langle E, F, H\rangle$. To do so, define the coordinate mapping $\Phi \colon \mathfrak{sl}_{2}(\mathbb{K}) \to \mathbb{K}^3$ such that for any linear combination $aE + bF + cH$, 
\[
    aE + bF + cH \mapsto (a, b, c)^T,
\]
\noindent
for all $a, b, c \in \mathbb{K}$. The bilinear left multiplication is defined by 
\[
    * \colon \mathfrak{sl}_{2}(\mathbb{K}) \times \mathfrak{sl}_{2}(\mathbb{K}) \to \mathfrak{sl}_{2}(\mathbb{K}), \ (X,Y) \mapsto X*Y,
\]

\noindent
for all $X,Y$ in $\mathfrak{sl}_{2}(\mathbb{K})$. Using Burde's method, we can represent the same left product through a coordinate representation. By fixing the basis $\langle E, F, H\rangle$, we identify $\mathfrak{sl}_{2}(\mathbb{K})$ via the linear coordinate isomorphism $\Phi$. For each $X \in \mathfrak{sl}_{2}(\mathbb{K})$, the left product
\[
    \lambda_{X} \colon \mathfrak{sl}_{2}(\mathbb{K}) \to \mathfrak{sl}_{2}(\mathbb{K}), \ \lambda_X(Y) \coloneq X*Y,
\]
defines a linear endomorphism of $\mathfrak{sl}_{2}(\mathbb{K})$ such that in coordinate representation, the left multiplication is given by
\[
    X*Y \coloneq \Phi^{-1}(\lambda_{X} \Phi(Y)).
\]

\noindent 
Indeed, for all $X, Y, Z \in \mathfrak{sl}_{2}(\mathbb{K})$, bilinearity gives $X*(aY +bZ) = a(X*Y) + b(X*Z)$ for all $a, b \in \mathbb{K}$.
\begin{remark}
    Note that $\mathfrak{sl}_{2}(\mathbb{K})$ and $\mathbb{K}^{3}$ are isomorphic as vector spaces. That is, $\Phi$ is only an isomorphism of vector spaces \cite{burde1}.
\end{remark}

\noindent
These relations can be summarized in the commutative diagram below.

\[
\begin{tikzcd}[row sep=large, column sep=huge]
\mathfrak{sl}_{2}(\mathbb{K}) \otimes \mathfrak{sl}_{2}(\mathbb{K})
  \arrow[r, "{\lambda \otimes \mathrm{id}}"]
  \arrow[dr, swap, "{*}"]
&
\mathrm{End}(\mathfrak{sl}_{2}(\mathbb{K}))\otimes \mathfrak{sl}_{2}(\mathbb{K}) \arrow[d, "{\mathrm{ev}}"]
\\
&
\mathfrak{sl}_{2}(\mathbb{K})
\end{tikzcd}
\]

\noindent
where $\mathrm{ev} \colon \mathrm{End}(\mathfrak{sl}_{2}(\mathbb{K})) \otimes \mathfrak{sl}_{2}(\mathbb{K}) \to \mathfrak{sl}_{2}(\mathbb{K})$ is the evaluation map defined by $T \otimes X \mapsto T(X)$, where $T$ is an endomorphism and $X \in \mathfrak{sl}_{2}(\mathbb{K})$.

Using these coordinate representations, we obtain the explicit products seen in Tables \ref{table1-reps} through \ref{table4-reps}. 

\begin{table}[H]
\centering
\begin{tabular}{cccc}
\hline
$*$ & $E$ & $F$ & $H$ \\
\hline
$E$ 
  & $H$ 
  & $uE + wF + (w^2+1)H$ 
  & $E - F - wH$ \\
$F$ 
  & $uE + wF + w^2H$ 
  & $(u - vw^2)E + vF - vwH$ 
  & $vwE + F - (u + v)H$ \\
$H$ 
  & $-F - wH$ 
  & $vwE - F - (u + v)H$ 
  & $-vE + H$ \\
\hline
\end{tabular} 
\caption{Explicit representation of Burde's left-symmetric structure $\mathcal{A}_{1,u,v}$ in terms of the $\langle E,F,H\rangle$ basis.}
\label{table1-reps}
\end{table}

\begin{table}[H]
\centering
\begin{tabular}{cccc}
\hline
$*$
 & $E$ & $F$ & $H$ \\
\hline
$E$ 
  & $0$ 
  & $(-1 -\gamma^{-1})H$ 
  & $\gamma E$ \\
$F$ 
  & $(1 - \gamma^{-1})H$ 
  & $0$ 
  & $\gamma F$ \\
$H$ 
  & $(\gamma - 1)E$ 
  & $(\gamma + 1)F$ 
  & $\gamma H$ \\
\hline
\end{tabular}
\caption{Explicit representation of Burde's left-symmetric structure $\mathcal{A}_{2,\gamma}$ in terms of the $\langle E,F,H\rangle$ basis.}
\label{table2-reps}
\end{table}

\begin{table}[H]
\centering
\begin{tabular}{cccc}
\hline
$*$
 & $E$ & $F$ & $H$ \\
\hline
$E$ 
  & $0$ 
  & $0$ 
  & $-E$ \\
$F$ 
  & $-H$ 
  & $E - H$ 
  & $E - F$ \\
$H$ 
  & $E$ 
  & $E$ 
  & $-H$ \\
\hline
\end{tabular}
\caption{Explicit representation of Burde's left-symmetric structure $\mathcal{A}_{3}$ in terms of the $\langle E,F,H\rangle$ basis.}
\label{table3-reps}
\end{table}

\begin{table}[H]
\centering
\begin{tabular}{cccc}
\hline
$*$
 & $E$ & $F$ & $H$ \\
\hline
$E$ 
  & $0$ 
  & $(\xi -1)H$ 
  & $\xi E$ \\
$F$ 
  & $(\xi + 1)H$ 
  & $E$ 
  & $\xi F$ \\
$H$ 
  & $(\xi - 1)E$ 
  & $(\xi + 1)F$ 
  & $\xi H$ \\
\hline
\end{tabular}
\caption{Explicit representation of Burde's left-symmetric structure $\mathcal{A}_4$ in terms of the $\langle E,F,H\rangle$ basis.}
\label{table4-reps}
\end{table}

Through these representations, we can verify whether $\mathfrak{sl}_{2}(\mathbb{K})$ admits Burde's structures. To do so, we must ensure that the Lie bracket associated with the left multiplication $[\cdot \mathbin{,} \cdot]_{*}$ recovers the original Lie bracket on $\mathfrak{sl}_{2}(\mathbb{K})$ itself. Furthermore, the structures must satisfy the left-symmetric identity seen in Eq. \eqref{lSID}. 

Indeed, we recover all original Lie bracket relations through the associated Lie bracket. Knowing that $\mathbb{K}$ is a field of characteristic 3, for the first structure $\mathcal{A}_{1, u, v}$ we have 
\begin{align*}
    [E,F]_{*} &= E*F - F*E = uE + wF + (w^2 + 1)H - (uE + wF + w^2H) = -2H = H, \\
    [H,E]_{*} &= H*E - E*H = -F -wH - (E-F-wH) = -E = 2E, \\
    [H,F]_{*} &= H*F - F*H = vwE -F -uH -vH -vwE -F +uH +vH = -2F.
\end{align*}
For the second structure $\mathcal{A}_{2, \gamma}$, we have the same result
\begin{align*}
    [E,F]_{*} &= E*F - F*E = (-1 - \gamma^{-1})H -(1 - \gamma^{-1})H = -2H = H, \\ [H,E]_{*} &= H*E - E*H = (\gamma - 1)E - \gamma E = -E = 2E, \\
    [H,F]_{*} &= H*F - F*H = (\gamma + 1)F - \gamma F = \gamma F + F - \gamma F = F = -2F.
\end{align*}
For the exceptional structures, we obtain the same result. For $\mathcal{A}_{3}$, we see that 
\begin{align*}
    [E,F]_{*} &= E*F - F*E = H, \\
    [H,E]_{*} &= H*E - E*H = E - (-E) = 2E, \\
    [H,F]_{*} &= H*F - F*H = E - (E - F) = F = -2F.
\end{align*}
For the fourth structure $\mathcal{A}_{4}$, we have
\begin{align*}
    [E,F]_{*} &= E*F - F*E = (\xi - 1)H - (\xi + 1)H = -2H = H, \\
    [H,E]_{*} &= H*E - E*H = (\xi - 1)E - \xi E = -E = 2E, \\
    [H,F]_{*} &= H*F - F*H = (\xi + 1)F - \xi F = F = -2F.
\end{align*}

Lastly, the left-symmetric identity is also fulfilled for all triples $X,Y,Z \in \langle E, F, H \rangle$. By multilinearity of the associator, this is enough to establish the identity for all elements $X, Y, Z \in \mathfrak{sl}_{2}(\mathbb{K})$. That is, we have that $(X,Y,Z) = (Y,X,Z)$, even after accounting for the characteristic of $\mathbb{K}$. Some example calculations are provided below for reference. 

For structure $\mathcal{A}_{1,u,v}$, we have the following example calculations
\begin{align*}
    (H,F,H) &= (H*F)*H - H*(F*H) \\ &= (vwE - F - (u + v)H)*H - H*(vwE + F - (u + v)H)\\ &= v(u-v)E + (2u + 2v)H,
\end{align*}
\noindent
and
\begin{align*}
    (F,H,H) &= (F*H)*H - F*(H*H) \\
    &= (vwE + F - (u + v)H)*H - F*(-vE + H) \\ &= v(u + 2v)E - (u+v)H.
\end{align*}

\noindent
We then have $(H,F,H) - (F,H,H) = -3v^{2}E + 3(u+v)H = 0$.

\noindent
Similarly, 
\begin{align*}
    (E,F,E) &= (E*F)*E - E*(F*E) \\
    &= (uE + wF + (w^2+1)H)*E - E*(uE + wF + w^2H) \\ &= -(u-v)^{2}E - F + (2u - 2v)H,
\end{align*}
\noindent
and
\begin{align*}
    (F,E,E) &= (F*E)*E - F*(E*E) \\ &= (uE + wF + w^2H)*E - F*H \\ &= -(u-v)^{2}E - F +(2u + v)H.
\end{align*}

\noindent
Hence, $(E,F,E) - (F,E,E) = -3vH = 0$.

For structure $\mathcal{A}_{2,\gamma}$, 
\begin{align*}
    (E, F, H) &= (E*F)*H - E*(F*H) = (-1 - \gamma^{-1})H * H - \gamma(E*F) \\ 
    &= (-1 - \gamma^{-1})\gamma H - (-1 - \gamma^{-1})\gamma H = 0,
\end{align*}
\noindent
and 
\begin{align*}
    (F,E,H) &= (F*E)*H - F*(E*H) = (1 - \gamma^{-1})H*H - \gamma(F*E) \\
    &= (1 - \gamma^{-1})\gamma H - (1 - \gamma^{-1})\gamma H = 0.
\end{align*}
\noindent
Therefore, $(E,F,H) = (F,E,H)$.

For structures $\mathcal{A}_{3}$ and $\mathcal{A}_{4}$, we have the same results. The example computations for these have been omitted as they are very similar to the previous structures. For the complete table with the respective calculations, see Appendix \ref{Appendix 1: Tables}.

\section{Superization of Burde's Left-Symmetric Structures} \label{Superization LSS}

We want to determine whether it is possible to extend Burde's structures to the Lie superalgebra $\mathfrak{osp}(1|2)$. Note that this Lie superalgebra is of superdimension $3|2$ with graded basis $\langle E, F, H \mid T_{+}, T_{-} \rangle$. Hence, we must define the left product $*$ such that

\begin{enumerate}
    \setlength{\itemsep}{0pt}   % Space between items
    \setlength{\parskip}{0pt}   % Space between paragraphs in items
    \setlength{\topsep}{0pt}
    \item [(i)] The supercommutator associated with the left multiplication, defined by 
    \begin{equation}
        [X,Y]_{*} = X*Y - (-1)^{\abs{X}\abs{Y}}Y*X
        \label{assoc_supercomm},
    \end{equation}
    \noindent for all homogeneous $X,Y \in \mathfrak{osp}(1|2)$ reproduces the Lie superalgebra $\mathfrak{osp}(1|2)$ relations in Eqs. \eqref{odd_rela} and \eqref{even_odd_rela};

    \item [(ii)] The $\mathbb{Z}/2\mathbb{Z}$-graded left-symmetric identity
    \begin{equation}
        (X, Y, Z) = (-1)^{\abs{X} \abs{Y}}(Y, X, Z),
        \label{left_sym_id}
    \end{equation}
    \noindent
    holds for all homogeneous $X, Y, Z \in \mathfrak{osp}(1|2)$. That is, the associator is supersymmetric in both $X$ and $Y$.
\end{enumerate}
\noindent
These are what we later define as Stages 1 and 2, respectively. 

Thus, define a similar coordinate isomorphism $\widehat{\Phi} \colon \mathfrak{osp}(1|2) \to \mathbb{K}^5$ sending to the standard basis in $\mathbb{K}^5$ relative to $\langle E, F, H \mid T_{+}, T_{-}\rangle$ such that for any linear combination $aE + bF + cH + dT_{+} + eT_{-}$,
\[
    aE + bF + cH + dT_{+} + eT_{-} \mapsto (a, b, c, d, e)^{T},
\]

\noindent
for all $a, b, c, d, e \in \mathbb{K}$.
For all homogeneous $X \in \mathfrak{osp}(1|2)$, define a representation 
\[
    \eta \colon \mathfrak{osp}(1|2) \to \mathrm{End}(\mathfrak{osp}(1|2)), \ X \mapsto \eta_{X}, \forall X \in \mathfrak{osp}(1|2).
\]

\noindent
The bilinear left multiplication is defined similarly to the normal case by 
\[
    * \colon \mathfrak{osp}(1|2) \times \mathfrak{osp}(1|2) \to \mathfrak{osp}(1|2), \ (X, Y) \mapsto X*Y,
\]

\noindent
for all $X,Y \in \mathfrak{osp}(1|2)$. Moreover, for all $Y \in \mathfrak{osp}(1|2)$, define the [superized] left multiplication operator by 
\[
    \eta_{X} \colon \mathfrak{osp}(1|2) \to \mathfrak{osp}(1|2), \ \eta_{X}(Y) \coloneq X*Y.
\]

\noindent
Once again, fixing a homogeneous basis and using the associated coordinate isomorphism $\widehat{\Phi}$, we can define the [superized] left multiplication in coordinate representation as follows
\[
    X * Y \coloneq \widehat{\Phi}^{-1}(\eta_{X} \widehat{\Phi}(Y)).
\]

\begin{remark}
    When restricted to the even part $\mathfrak{sl}_{2}(\mathbb{K})$, the product $*$ defines an ordinary left-symmetric algebra. When extended to $\mathfrak{osp}(1|2)$, this product defines a left-symmetric superalgebra product. All identities here are understood in the graded sense. 
\end{remark}

Relative to the decomposition $\mathfrak{osp}(1|2) = \mathfrak{sl}_2(\mathbb{K}) \oplus \mathfrak{osp}(1|2)_{\bar1}$, every left-symmetric superalgebra in $M_{5}(\mathbb{K})$ can be understood as $2 \times 2$ block matrices of $(3+2)\times(3+2)$ dimension
\begin{equation}
    \eta_{X} \cong \begin{pmatrix}
        A_{X} & B_{X} \\ C_{X} & D_{X}
    \end{pmatrix}, \label{osp-ref}
\end{equation}

\noindent
such that $A_{X}, B_{X}, C_{X}, \ \text{and} \ D_{X}$ are $3\times3$, $3\times2$, $2\times3$, and $2\times2$ endomorphisms, respectively, where

\begin{enumerate}
    \setlength{\itemsep}{0pt}   % Space between items
    \setlength{\parskip}{0pt}   % Space between paragraphs in items
    \setlength{\topsep}{0pt}
    \item [(1)] $A_{X} \colon \mathfrak{sl}_2(\mathbb{K}) \to \mathfrak{sl}_2(\mathbb{K})$,
    \item [(2)] $B_{X} \colon \mathfrak{osp}(1|2)_{\bar1} \to \mathfrak{sl}_2(\mathbb{K})$,
    \item [(3)] $C_{X} \colon \mathfrak{sl}_2(\mathbb{K}) \to \mathfrak{osp}(1|2)_{\bar1}$,
    \item [(4)] $D_{X} \colon \mathfrak{osp}(1|2)_{\bar1} \to \mathfrak{osp}(1|2)_{\bar1}$.
\end{enumerate}

\begin{remark}
    Note that the relationship between $\eta_{X}$ and the $2 \times 2$ block matrix in Eq. \eqref{osp-ref} is an isomorphism of vector spaces. Letting $\mathfrak{g} = \mathfrak{osp}(1|2)$, the direct sum decomposition of $\mathfrak{g}$ induces the identification
    \[
    \mathrm{End}(\mathfrak{g}_{\bar{0}} \oplus \mathfrak{g}_{\bar{1}}) \cong \begin{pmatrix}
        \mathrm{Hom}(\mathfrak{g}_{\bar{0}}, \mathfrak{g}_{\bar{0}}) & \mathrm{Hom}(\mathfrak{g}_{\bar{1}}, \mathfrak{g}_{\bar{0}}) \\ 
        \mathrm{Hom}(\mathfrak{g}_{\bar{0}}, \mathfrak{g}_{\bar{1}}) &  \mathrm{Hom}(\mathfrak{g}_{\bar{1}}, \mathfrak{g}_{\bar{1}})
    \end{pmatrix},
    \]

    \noindent
    where each endomorphism of $\mathfrak{g}$ is uniquely determined by its four components mapping between the even and odd subspaces. These precisely correspond to the $A,B,C,D$ endomorphisms. 
\end{remark}

For the even generators $\langle E, F, H\rangle$, the $A$-matrices are already given by Burde in \cite{burde-thesis}. However, we must find the corresponding $D$-matrices describing the action of the even generators on the odd subspace. For the odd generators $\langle T_{+}, T_{-}\rangle$, each $\eta_{T_{\pm}}$ has off-diagonal blocks $B_{T{_{\pm}}} \in M_{3\times2}(\mathbb{K})$ and $C_{T{_{\pm}}} \in M_{2\times3}(\mathbb{K})$, which encode the odd-odd and odd-even products. These are also unknown. Hence, we must find the respective entries of
\begin{enumerate}
    \setlength{\itemsep}{0pt}   % Space between items
    \setlength{\parskip}{0pt}   % Space between paragraphs in items
    \setlength{\topsep}{0pt}
    \item [(1)] $D$-blocks for $\eta_{E}, \eta_{F}, \eta_{H}$, each with 4 entries;
    \item [(2)] $B$- and $C$-blocks for $\eta_{T_{_{\pm}}}$, each with 6 entries.
\end{enumerate}

\subsection{Stage 1: The Associated Supercommutator.} \label{Computation}

At this stage, the relations between even and odd generators are unknown. Thus, to apply the supercommutator associated with the left multiplication and establish the relation between even and odd generators, we introduce constant structures for all left products involving odd generators. We fix the $\langle E,F,H \mid T_{+}, T_{-} \rangle$ basis. For the odd-odd products, let

\[
    \begin{cases}
    T_{+} * T_{+} = \alpha_{++}^{E}E + \alpha_{++}^{F}F + \alpha_{++}^{H}H \\ 
    T_{-} * T_{-} = \alpha_{--}^{E}E + \alpha_{--}^{F}F + \alpha_{--}^{H}H \\ 
    T_{+} * T_{-} = \alpha_{+-}^{E}E + \alpha_{+-}^{F}F + \alpha_{+-}^{H}H \\ 
    T_{-} * T_{+} = \alpha_{-+}^{E}E + \alpha_{-+}^{F}F + \alpha_{-+}^{H}H. 
    \end{cases}
\]
Similarly, for the odd-even products 
\[
    \begin{cases}
    T_{+} * E = \beta_{+E}^{+}T_{+} + \beta_{+E}^{-}T_{-} \\
    T_{+} * F = \beta_{+F}^{+}T_{+} + \beta_{+F}^{-}T_{-} \\
    T_{+} * H = \beta_{+H}^{+}T_{+} + \beta_{+H}^{-}T_{-} \\
    T_{-} * E = \beta_{-E}^{+}T_{+} + \beta_{-E}^{-}T_{-} \\
    T_{-} * F = \beta_{-F}^{+}T_{+} + \beta_{-F}^{-}T_{-} \\
    T_{-} * H = \beta_{-H}^{+}T_{+} + \beta_{-H}^{-}T_{-}. 
    \end{cases}
\]
Lastly, for the even-odd products, 

\[
    \begin{cases}
    E * T_{+} = \delta_{E+}^{+}T_{+} + \delta_{E+}^{-}T_{-} \\
    E * T_{-} = \delta_{E-}^{+}T_{+} + \delta_{E-}^{-}T_{-} \\
    F * T_{+} = \delta_{F+}^{+}T_{+} + \delta_{F+}^{-}T_{-} \\
    F * T_{-} = \delta_{F-}^{+}T_{+} + \delta_{F-}^{-}T_{-} \\
    H * T_{+} = \delta_{H+}^{+}T_{+} + \delta_{H+}^{-}T_{-} \\
    H * T_{-} = \delta_{H-}^{+}T_{+} + \delta_{H-}^{-}T_{-}.
    \end{cases}
\]

For example, if we consider the endomorphisms $D_{E} \in M_{2}(\mathbb{K}),  \ B_{T_{+}} \in M_{3 \times 2}(\mathbb{K})$ and $C_{T_{-}} \in M_{2 \times 3}(\mathbb{K})$, we have
\begin{equation}
    D_E = \begin{pmatrix} \delta_{E+}^{+} & \delta_{E+}^{-} \\
    \delta_{E-}^{+} & \delta_{E-}^{-}
    \end{pmatrix}, \ 
    B_{T_{+}} = \begin{pmatrix} \alpha_{++}^{E} & \alpha_{+-}^{E} \\ 
    \alpha_{++}^{F} & \alpha_{+-}^{F} \\ \alpha_{++}^{H} & \alpha_{+-}^{H} \end{pmatrix}, \ C_{T_{-}} = \begin{pmatrix}
        \beta_{-E}^{+} & \beta_{-F}^{+} & \beta_{-H}^{+} \\ \beta_{-E}^{-} & \beta_{-F}^{-} & \beta_{-H}^{-}
    \end{pmatrix}.
\end{equation}

Having introduced constants for the odd and even relations, we construct a linear system of equations using the associated supercommutator. In building this system, it suffices to impose the relation for one orientation of each unordered pair $(X,Y)$ for all homogeneous $X, Y \in \mathfrak{osp}(1|2)$, with the associated supercommutator. This is because the relation for the unordered pair $(Y,X)$ follows immediately from skew-supersymmetry. Furthermore, all nonzero products in Eqs. \eqref{odd_rela} and \eqref{even_odd_rela} are considered, but with the associated supercommutator. That is, using the product in Eq. \eqref{assoc_supercomm}, we plug in the basis pairs and equate the coefficients with the results in Eqs. \eqref{odd_rela} and \eqref{even_odd_rela}. We present an example computation of this process below.

From $\{T_{+}, T_{+}\}=2E$:
\begin{align*}
    \{T_{+}, T_{+}\}_{*} \stackrel{!}{=} 2E &\implies 2(T_{+}*T_{+}) = 2E \\ 
    &\implies 2(\alpha_{++}^{E}E + \alpha_{++}^{F}F + \alpha_{++}^{H}H) = 2E \\
    &\implies \alpha_{++}^{E} = 1, \alpha_{++}^{F} = 0, \alpha_{++}^{H} = 0.
\end{align*}

We continue this process iteratively for all nonzero pairs. Afterwards, all equations are simplified and solved for all coefficients. This process is carried out for all structures in Eqs. \eqref{lss1}, \eqref{lss2} \eqref{lss3}, and \eqref{lss4}. For a total of 45 scalar equations with 36 unknown coefficients, a single solution set was found. The system is consistent with 21 coefficients determined (parametric or non-parametric), and 15 free coefficients. Table \ref{dependent_params_table} shows the parameterization of the 21 dependent unknowns for all structures. The remaining coefficients correspond to free unknowns. Interestingly, all four structures share the same parametrizations, as the solution space of bilinear products associated with the left multiplication is independent of the choice of the left-symmetric algebra in the even subalgebra $\mathfrak{osp}(1|2)_{\bar{0}}$.

\begin{table}[!ht]
\centering
\begin{tabular}{cc}
\hline
Coefficient & Expression \\ \hline
$\delta_{E+}^{+}$ & $\beta_{+E}^{+}$ \\
$\delta_{E+}^{-}$ & $\beta_{+E}^{-}$ \\
$\delta_{E-}^{+}$ & $\beta_{-E}^{+} - 1$ \\
$\delta_{E-}^{-}$ & $\beta_{-E}^{-}$ \\
$\delta_{F+}^{+}$ & $\beta_{+F}^{+}$ \\
$\delta_{F+}^{-}$ & $\beta_{+F}^{-} - 1$ \\
$\delta_{F-}^{+}$ & $\beta_{-F}^{+}$ \\
$\delta_{F-}^{-}$ & $\beta_{-F}^{-}$ \\
$\delta_{H+}^{+}$ & $\beta_{+H}^{+} + 1$ \\
$\delta_{H+}^{-}$ & $\beta_{+H}^{-}$ \\
$\delta_{H-}^{+}$ & $\beta_{-H}^{+}$ \\
$\delta_{H-}^{-}$ & $\beta_{-H}^{-} - 1$ \\
$\alpha_{++}^{E}$ & $1$ \\
$\alpha_{++}^{F}$ & $0$ \\
$\alpha_{++}^{H}$ & $0$ \\
$\alpha_{+-}^{E}$ & $-\alpha_{-+}^{E}$ \\
$\alpha_{+-}^{F}$ & $-\alpha_{-+}^{F}$ \\
$\alpha_{+-}^{H}$ & $1 - \alpha_{-+}^{H}$ \\
$\alpha_{--}^{E}$ & $0$ \\
$\alpha_{--}^{F}$ & $-1$ \\
$\alpha_{--}^{H}$ & $0$ \\ \hline
\end{tabular}
\caption{Solutions for the coefficients expressed in terms of the free parameters.}
\label{dependent_params_table}
\end{table}

\begin{remark}
    Stage 2 of the superization---that is, the verification of the $\mathbb{Z}/2\mathbb{Z}$-graded left-symmetric identity---is not independent of the choice of left-symmetric structure. This is because the left-symmetric identity is quadratic in the constants and thus introduces nonlinearities. Moreover, the associator contains nested products; if $X,Y,Z \in \mathfrak{osp}(1|2)_{\bar{0}}$ then all products of the associator $(X,Y,Z)$ are already determined by Burde in \cite{burde1}.
\end{remark}

\subsection{Stage 2: The \texorpdfstring{$\mathbb{Z}/2\mathbb{Z}$}{Z/2Z}-Graded Left-Symmetric Identity.}\label{Stage-2}

Having fulfilled Stage 1 of the superization, we must now determine if the $\mathbb{Z}/2\mathbb{Z}$-graded left-symmetric identity holds for all homogeneous $X,Y,Z \in \mathfrak{osp}(1|2)$. To do so, we employ the associator endowed with the left multiplication. Given that $\mathfrak{osp}(1|2)$ has $\mathrm{sdim}=3|2$, Stage 2 of the superization produces a system of 625 equations for each left-symmetric structure. For structure $\mathcal{A}_{1,u,v}$, we obtain a total of 225 nonzero polynomials, of which 22 are linear polynomial equations and 203 are nonlinear polynomial equations. Subsequently, we substitute the solutions in Table \ref{dependent_params_table} from the system of linear equations imposed by the associated supercommutator. For structure $\mathcal{A}_{1,u,v}$, we obtain the following system of [nontrivial] nonlinear polynomial equations
\begin{equation}
    \{p_i(\tau_{1},\ldots,\tau_{15}) =0\}_{i=1}^{203},
    \label{sys-1}
\end{equation} 
\noindent
such that $p_{i} \in \mathbb{K}[u,v][\tau_{1},\ldots,\tau_{15}]$ for $i = 1,\ldots, 203,$ where $u,v$ correspond to the same constants of structure $\mathcal{A}_{1,u,v}$ in Eq. \eqref{lss1}.

We approach structure $\mathcal{A}_{2,\gamma}$ using the same method. We obtain a total of 207 nonzero polynomial equations, of which 4 are linear polynomial equations and 203 are nonlinear polynomial equations. After substitution of the solutions from the previous linear system of equations, we obtain the following system of [nontrivial] nonlinear polynomial equations
\begin{equation}
    \{q_i(\tau_{1},\ldots,\tau_{15}) = 0\}_{i=1}^{203},
    \label{sys-2}
\end{equation}

\noindent
where $q_{i} \in \mathbb{K}[\gamma, \gamma^{-1}][\tau_{1},\ldots,\tau_{15}]$ for $i = 1 ,\ldots, 203$, where $\gamma \in \mathbb{K}^{\times}$ corresponds to the parameter in Eq. \eqref{lss2}.

For structure $\mathcal{A}_{3}$, we obtain a total of 205 nonzero polynomial equations, of which 2 are linear and 203 are nonlinear polynomial equations. Substituting the same solutions from the linear system of equations, we obtain the following system
\begin{equation}
    \{r_{i}(\tau_{1},\ldots,\tau_{15}) = 0\}_{i = 1}^{203},
    \label{sys-3}
\end{equation}

\noindent
where $r_{i} \in \mathbb{K}[\tau_{1},\ldots,\tau_{15}]$ for $i = 1,\ldots, 203$. 

Lastly, for structure $\mathcal{A}_{4}$, we obtain 207 nonzero polynomial equations, of which 4 are linear and 203 are nonlinear. Once again, after having substituted the solutions of Table \ref{dependent_params_table} into this polynomial system, we obtain 
\begin{equation}
    \{s_{i}(\tau_{1}, \ldots, \tau_{15}) = 0 \}_{i=1}^{203},
    \label{sys4}
\end{equation}

\noindent
such that $s_{i} \in \mathbb{K}[\xi][\tau_{1}, \ldots, \tau_{15}]$ for $i=1,\ldots,203$, where $\xi^{2} = -1$.

\begin{remark}
 The variables $\tau_{i}$ in each system correspond to the remaining 15 free structure constants from Stage 1. We express them in this way for convenience. Such a correspondence is presented in Table \ref{tau_correspondence}. 
\end{remark}

\begin{table}[H]
\centering
\setlength{\tabcolsep}{10pt}
\begin{tabular}{ccc}
\hline
Variable & Code symbol & Coefficient \\ \hline
$\tau_1$ & \texttt{b\_Tm\_E\_Tm} & $\beta_{-E}^{-}$ \\ 
$\tau_2$ & \texttt{b\_Tm\_E\_Tp} & $\beta_{-E}^{+}$ \\ 
$\tau_3$ & \texttt{b\_Tm\_F\_Tm} & $\beta_{-F}^{-}$ \\ 
$\tau_4$ & \texttt{b\_Tm\_F\_Tp} & $\beta_{-F}^{+}$ \\ 
$\tau_5$ & \texttt{b\_Tm\_H\_Tm} & $\beta_{-H}^{-}$ \\ 
$\tau_6$ & \texttt{b\_Tm\_H\_Tp} & $\beta_{-H}^{+}$ \\
$\tau_7$ & \texttt{b\_Tm\_Tp\_E} & $\alpha_{-+}^{E}$ \\ 
$\tau_8$ & \texttt{b\_Tm\_Tp\_F} & $\alpha_{-+}^{F}$ \\ 
$\tau_9$ & \texttt{b\_Tm\_Tp\_H} & $\alpha_{-+}^{H}$ \\ 
$\tau_{10}$ & \texttt{b\_Tp\_E\_Tm} & $\beta_{+E}^{-}$ \\ 
$\tau_{11}$ & \texttt{b\_Tp\_E\_Tp} & $\beta_{+E}^{+}$ \\ 
$\tau_{12}$ & \texttt{b\_Tp\_F\_Tm} & $\beta_{+F}^{-}$ \\ 
$\tau_{13}$ & \texttt{b\_Tp\_F\_Tp} & $\beta_{+F}^{+}$ \\ 
$\tau_{14}$ & \texttt{b\_Tp\_H\_Tm} & $\beta_{+H}^{-}$ \\ 
$\tau_{15}$ & \texttt{b\_Tp\_H\_Tp} & $\beta_{+H}^{+}$ \\ \hline
\end{tabular}
\caption{Correspondence between the polynomial indeterminates $\tau_i$, $i=1,\ldots,15$, code symbols, and free structure constants. Here, the code symbol refers to the notation employed in the code.}
\label{tau_correspondence}
\end{table}

By Section \ref{Grobner bases}, we know that each system of [nonlinear] polynomial equations determines the ideals
\begin{equation}
   I_{1} = \langle p_{1} ,\ldots, p_{203} \rangle \subset \mathbb{K}[u,v][\tau_{1},\ldots,\tau_{15}],
\end{equation}
\begin{equation}
    I_{2} = \langle q_{1} ,\ldots, q_{203} \rangle \subset \mathbb{K}[\gamma, \gamma^{-1}][\tau_{1},\ldots,\tau_{15}],
\end{equation}
\begin{equation}
    I_{3} = \langle r_{1}, \ldots, r_{203} \rangle \subset \mathbb{K}[\tau_{1},\ldots,\tau_{15}],
\end{equation}
\begin{equation}
    I_{4} = \langle s_{1}, \ldots, s_{203} \rangle \subset \mathbb{K}[\xi][\tau_{1},\ldots,\tau_{15}].
\end{equation}

We need to determine whether each system of equations has a solution in its corresponding polynomial ring to determine if the left-symmetric identity is satisfied. We proceed with two approaches for determining if each system has a solution of the form $(\tau_{1},\ldots,\tau_{15}) \in \mathbb{K}^{15}$. We compute the Gröbner basis for each polynomial system, and as a complementary check, we proceed with a 2-level polynomial factorization method. The code for Stage 2 can be found in the \verb|superization.py| file in the GitHub repository. A second code script was written using the SageMath Python library to verify the Gröbner basis computation. 

\subsubsection{Gröbner Basis Computation.} \label{grobner-basis-comp}

The first approach entails solving each system of equations over the field $\mathrm{GF}(3)$ by computing a Gröbner basis for each ideal $I_j$ for $j \in \{1,2,3,4\}$. We use the Gröbner basis function found in the Python SymPy library. Here, $\mathrm{GF}(3)$ corresponds to the Galois Field of order 3, or in other words, the finite field of 3 elements, equivalently expressed as $\mathbb{F}_{3} = \{0,1,2\}$. Moreover, we use Buchberger's algorithm established in \cite{buchberger-grobner} as the underlying method, and impose a lexicographical term order $\succ$ on the polynomials such that $\tau_{1} \succ \tau _{2} \succ \cdots \succ \tau_{14} \succ \tau_{15}$.  

Naturally, the choice of domain might seem questionable, particularly because the underlying vector space of $\mathfrak{osp}(1|2)$ lies over an arbitrary field $\mathbb{K}$ of characteristic 3. One might assume that our results in $\mathbb{F}_{3}$ do not hold in any field of characteristic 3. However, note that $\mathbb{F}_{3}$ is of characteristic 3 itself. Later in this section, we will justify how this computation generalizes results over all fields of characteristic 3.

In \cite{burde1}, the parameters $u,v \in \mathbb{K}$, and $\gamma \in \mathbb{K}^{\times}$ for structures $\mathcal{A}_{1,u,v}$ and $\mathcal{A}_{2,\gamma}$, respectively, are given arbitrarily. SymPy's Gröbner basis function does not allow us to specify arbitrary fields of characteristic $p \in \mathbb{Z}^{+}$. In order to account for this, we treat these parameters as additional polynomial unknowns alongside the rest of the $\tau_{i}$. Hence, the redefined ideals for the polynomial systems in Eqs. \eqref{sys-1} and \eqref{sys-2} are given by
\begin{equation}
    I_{1} \subset \mathbb{F}_{3}[\tau_{1},\ldots,\tau_{15}, u, v],
\end{equation}
\begin{equation}
    I_{2} = \langle q_{1},\ldots,q_{203}, \gamma \cdot \gamma^{-1} -1 \rangle \subset \mathbb{F}_{3}[\tau_{1},\ldots,\tau_{15}, \gamma, \gamma^{-1}].
\end{equation}

\noindent
Doing this avoids specializing these parameters to any value in $\mathbb{F}_{3}$. Additionally, we must point out that Gröbner basis algorithms require ordinary polynomial rings. This means that for $I_{2}$, these algorithms cannot handle $\gamma^{-1}$ directly, as the division algorithm relies on the term orders mentioned earlier in Section \ref{Grobner bases}, which require non-negative exponents. Thus, we introduce the indeterminate $\gamma^{-1}$ and enforce invertibility by adjoining the relation $\gamma \cdot \gamma^{-1} - 1$ as a generator of the ideal. This trivially enforces the condition $\gamma \neq 0$ and restricts the variety $V(I_{2})$ to the hypersurface $\gamma \cdot \gamma^{-1} =1$.

For structure $\mathcal{A}_{4}$, we employ a similar logic. The algorithm treats $\xi$ as a symbol instead of letting $\xi = i \in \mathbb{C}$, especially because $\xi \notin \mathbb{F}_{3}$. Thus, we directly introduce the relation $\xi^2 + 1 = 0$ and enforce it by adjoining it as a generator of the ideal. Therefore, the redefined ideals for the polynomial systems in Eqs. \eqref{sys-3} and \eqref{sys4} are defined by
\begin{equation}
    I_3 = \langle r_{1}, \ldots, r_{203} \rangle \subset \mathbb{F}_{3}[\tau_{1},\ldots,\tau_{15}], \text{ and}
\end{equation}
\begin{equation}
    I_{4} = \langle s_{1},\ldots,s_{203}, \xi^{2} + 1 \rangle \subset \mathbb{F}_{3}[\tau_{1},\ldots,\tau_{15}, \xi].
\end{equation}

Using the specified field, term order, and ideals mentioned earlier, for all extended structures, we obtain a Gröbner basis $G = \{1\}$. By Section \ref{Grobner bases}, $1 \in I_{j}$ if and only if $I_{j} = \langle 1 \rangle$ for all $j \in \{1,2,3,4\}$. Hence, all systems of equations are inconsistent in the field $\mathbb{F}_{3}$. We obtain the same results by computing the Gröbner basis for all four ideals using the field extension $\mathbb{F}_{9}$. Before proceeding to the generalization of these results for arbitrary field extensions, consider the following definitions.

\begin{definition}
\normalfont
Let $\mathbb{K}$ be a field. A subset $F$ that is itself a field under the operations of $\mathbb{K}$ is called a \textit{subfield} of $\mathbb{K}$. The field $\mathbb{K}$ is called an \textit{extension field} of $F$. If $F \neq \mathbb{K}$, $F$ is called a \textit{proper} subfield of $\mathbb{K}$. 
\end{definition}

Now, a field containing no proper subfields is called a \textit{prime} field. The intersection of all subfields of a field $\mathbb{K}$ is itself a field, called the \textit{prime subfield} of $\mathbb{K}$. Note that the prime subfield of $\mathbb{K}$ is a field, trivially. 

We introduce the following theorem relevant to our discussion on Gröbner bases. 

\begin{theorem}
    Let $\mathbb{K}$ be a field. The prime subfield of $\mathbb{K}$ is isomorphic to $\mathbb{Q}$ if $\mathbb{K}$ has characteristic $0$, and is isomorphic to $\mathbb{F}_{p}$ if $\mathbb{K}$ has characteristic $p$.
\end{theorem}\label{prime-subfield-thm}

\noindent
By the previous theorem, we have that if $\mathbb{K}$ is an arbitrary field of characteristic 3, then it contains a prime subfield isomorphic to $\mathbb{F}_{3}$. Thus, $\mathbb{F}_{3} \subset \mathbb{K}$. Hence, there exists a canonical embedding given by the inclusion map

\begin{equation}
    \iota : \mathbb{F}_{3} \hookrightarrow \mathbb{K},
    \label{canonical-embedding} 
\end{equation}

\noindent
for every field $\mathbb{K}$ of characteristic 3.

Our Gröbner basis computation for all LSSAs on the base field $\mathbb{F}_{3}$ implies that there exists a Nullstellensatz certificate $\{h_{i}\}_{i=1}^{203}$ such that

\begin{equation} \sum\limits_{i=1}^{203}h_{i}f_{i} = 1,
    \label{certificate-nullstellen}
\end{equation}

\noindent
in the respective polynomial rings $\mathbb{F}_{3}[\tau_{1},\ldots,\tau_{15}, u, v]$,  $\mathbb{F}_{3}[\tau_{1},\ldots,\tau_{15}, \gamma, \gamma^{-1}]$, $\mathbb{F}_{3}[\tau_{1},\ldots,\tau_{15}]$, and $\mathbb{F}_{3}[\tau_{1},\ldots,\tau_{15}, \xi]$. Let $\mathbb{K}$ be any field of characteristic 3. Since we know that $\mathbb{F}_{3} \subset \mathbb{K}$, the embedding in \eqref{canonical-embedding} induces a ring homomorphism for each structure. For structure $\mathcal{A}_{1,u,v}$, this homomorphism is defined by 
\begin{equation}
    \tilde{\iota}_{1} : \mathbb{F}_{3}[\tau_{1},\ldots, \tau_{15}, u, v] \hookrightarrow \mathbb{K}[\tau_{1},\ldots, \tau_{15}, u, v], 
    \label{poly-ring-embed-1}
\end{equation}

\noindent
For structures $\mathcal{A}_{2,\gamma}, \mathcal{A}_{3}$, and $\mathcal{A}_{4}$, we define such homomorphisms similarly,
\begin{equation}
    \tilde{\iota}_{2} : \mathbb{F}_{3}[\tau_{1},\ldots, \tau_{15}, \gamma, \gamma^{-1}] \hookrightarrow \mathbb{K}[\tau_{1},\ldots, \tau_{15}, \gamma, \gamma^{-1}],
    \label{poly-ring-embed-2}
\end{equation}
\begin{equation}
    \tilde{\iota}_{3} : \mathbb{F}_{3}[\tau_{1},\ldots, \tau_{15}] \hookrightarrow \mathbb{K}[\tau_{1},\ldots, \tau_{15}],
    \label{poly-ring-embed-3}
\end{equation}
\begin{equation}
    \tilde{\iota}_{4} : \mathbb{F}_{3}[\tau_{1},\ldots, \tau_{15}, \xi] \hookrightarrow \mathbb{K}[\tau_{1},\ldots, \tau_{15}, \xi].
    \label{poly-ring-embed-4}
\end{equation}

\noindent
In all homomorphisms, the constant polynomial $1$ maps to itself, and nonzero polynomials are mapped to nonzero polynomials. Following Atiyah and Macdonald's Chapter 1.7 \textit{Extension and Contraction} in \cite{atiyah-macdonald}, given a ring homomorphism $\phi \colon A \to B$ and an ideal $J \subset A$, the extension $J^{B}$ of $J$ is defined as the ideal of $B$ generated by $\phi(J)$. Explicitly, the extension is defined as $J^{B} \coloneqq B \cdot \phi(J)$ and is the set of all sums given by
\[
    \sum y_{i}\phi(x_{i}),
\]
\noindent
such that $x_{i} \in J, \ y_{i} \in B$.
In our setting, the ring homomorphisms in Eqs. \eqref{poly-ring-embed-1}, \eqref{poly-ring-embed-2}, \eqref{poly-ring-embed-3}, and \eqref{poly-ring-embed-4} apply $\iota$ to each coefficient and fix each indeterminate. Hence, if $I_{1} \subset \mathbb{F}_{3}[\tau_{1},\ldots, \tau_{15}, u, v]$ is an ideal, then the extension $I_{1}^{\mathbb{K}}$ is the ideal of $\mathbb{K}[\tau_{1},\ldots, \tau_{15}, u, v]$ generated by $\tilde{\iota}_{1}(I_{1})$ such that
\begin{equation}
    I_{1}^{\mathbb{K}} \coloneqq \mathbb{K}[\tau_{1}, \ldots, \tau_{15}, u, v] \cdot \tilde{\iota}_{1}(I_{1}) = \langle \tilde{\iota}_{1}(p_{1}), \ldots, \tilde{\iota}_{1}(p_{203}) \rangle \subset \mathbb{K}[\tau_{1},\ldots, \tau_{15}, u, v].
\end{equation}

\noindent
Since $\iota$ is just the inclusion map, the generators of $I_{1}^{\mathbb{K}}$ are the same polynomials as those of $I_{1}$ viewed as elements of $\mathbb{K}[\tau_{1},\ldots, \tau_{15}, u, v]$. The same logic follows for ideals $I_{2}, I_3$ and $I_4$ and their extensions.

By our computation, if $1 \in I_{j}$ for all $j \in \{1,2,3,4\}$, then by the previous homomorphism, $1 \in I_{j}^{\mathbb{K}}$ for all $j \in \{1,2,3,4\}$ as well. As a result, the identity in \eqref{certificate-nullstellen} remains true in the field extension $\mathbb{K}$. This result is summarized in Lemma \ref{scalar-extensions}.

\begin{remark}
    We must highlight that, while we introduced Hilbert's Weak Nullstellensatz as a motivation for Section \ref{Grobner bases}, we do not require an algebraically closed field for our argument. Namely, we only justify the trivial direction $1 \in I \implies 1 \in I^{\mathbb{K}}$ such that $\mathbb{F}_{3} \subset \mathbb{K}$ and $\mathbb{K}$ is any field of characteristic 3.
\end{remark}

\begin{lemma}
\label{scalar-extensions}    
Let $F$ be a field, $I = \langle f_{1}, \ldots, f_{m} \rangle \subset F[x_{1}, \ldots, x_{n}]$ an ideal, and let $\mathbb{K}$ be a field extension such that $F \subset \mathbb{K}$. Define the extended ideal $I^{\mathbb{K}} \coloneqq I \cdot \mathbb{K}[x_1, \ldots, x_n] = \langle \tilde{\iota}(f_{1}), \ldots, \tilde{\iota}(f_{m}) \rangle \subset \mathbb{K}[x_1, \ldots, x_n].$ If $1 \in I$, then $1 \in I^{\mathbb{K}}$. In particular, $V_{\mathbb{K}}(I^{\mathbb{K}}) = \varnothing$. 
\end{lemma}

\begin{proof}
    The inclusion $\iota : F \hookrightarrow \mathbb{K}$ induces the injective ring homomorphism
    \[
        \tilde{\iota}: F[x_{1}, \ldots, x_{n}] \hookrightarrow \mathbb{K}[x_{1}, \ldots, x_{n}]
    \] 

    \noindent
    that fixes each indeterminate and extends the coefficients. Since $1 \in I$, there exist polynomials $h_1, \ldots, h_m \in F[x_{1}, \ldots, x_{n}]$ such that 
    \[
        \sum\limits_{i=1}^{m}h_{i}f_{i} = 1.
    \]

    \noindent
    Applying $\tilde{\iota}$ to the previous identity yields 
    \[ 
       \sum\limits_{i=1}^{m} \tilde{\iota}(h_{i}) \tilde{\iota}(f_{i}) = 1 \in \mathbb{K}[x_{1}, \ldots, x_{n}].
    \]

    \noindent
    Since each $\tilde{\iota}(f_{i}) = f_{i}$, $i=1, \ldots,m$, viewed in $\mathbb{K}[x_{1}, \ldots, x_{n}]$, we have $1 \in I^{\mathbb{K}}$.

    Now, suppose that ${\bf{x}} = (x_{1},\ldots,x_{n}) \in \mathbb{K}^{n}$ is a common zero of $I^{\mathbb{K}}$. Then, evaluating at $\bf{x}$, we have  
    \[
      \sum\limits_{i=1}^{m} \tilde{\iota}(h_{i}({\bf{x}})) f_{i}({\bf{x}}) = 1  \implies 0=1,
    \]

    \noindent
    which is impossible. Therefore, no solution exists in $\mathbb{K}^{n}$ and thus $V_{\mathbb{K}}(I^{\mathbb{K}}) = \varnothing$.
\end{proof}

By Lemma \eqref{scalar-extensions}, the systems of polynomial equations in \eqref{sys-1}, \eqref{sys-2}, \eqref{sys-3}, and \eqref{sys4} are inconsistent over any field $\mathbb{K}$ of characteristic 3. Therefore, Burde's structures cannot be extended to $\mathfrak{osp}(1|2)$, and thus this Lie superalgebra does not admit left-symmetric superalgebra structures. We can summarize this result in the following theorem.

\begin{theorem}[Left-Symmetric Structures on $\mathfrak{osp}(1|2)$]
    Let $\mathbb{K}$ be a field of characteristic different from $2$. Then, the Lie superalgebra $\mathfrak{osp}(1|2)$ does not admit a compatible left-symmetric superalgebra structure. 
\end{theorem}

\begin{proof}
    Suppose the Lie superalgebra $\mathfrak{osp}(1|2)$ admits a compatible left-symmetric superalgebra structure over $\mathbb{K}$. Then, its restriction to the even part $\mathfrak{osp}(1|2)_{\bar{0}}$ defines a compatible left-symmetric algebra structure on $\mathfrak{osp}(1|2)_{\bar{0}} \cong \mathfrak{sl}_{2}(\mathbb{K})$.

    In characteristic $0$, such a structure on $\mathfrak{sl}_{2}(\mathbb{K})$ is ruled out by Helmstetter's standard nonexistence result for perfect Lie algebras \cite{helmstetter}.

    In characteristic $p >3$, Burde's result (established in Theorem \ref{burde-theorem-specialized}) rules it out because $\mathrm{dim}(\mathfrak{sl}_{2}(\mathbb{K})) = 3$ and $p$ does not divide $3$.

    It remains to consider the case of characteristic $3$. By changing the base field, the assumed left-symmetric superalgebra over $\mathbb{K}$ extends to a left-symmetric structure over the algebraic closure $\overline{\mathbb{K}}$, with the even part remaining a left-symmetric algebra on $\mathfrak{sl}_{2}(\mathbb{K})$. Burde classifies all compatible left-symmetric algebra structures on $\mathfrak{sl}_{2}(\mathbb{K})$ over algebraically closed fields into the four families $\mathcal{A}_{1,u,v}$, $\mathcal{A}_{2, \gamma}$, $\mathcal{A}_{3}$, and $\mathcal{A}_{4}$. The Gröbner basis computations in Section \ref{grobner-basis-comp} show that, for each of the four structures, the resulting polynomial system is inconsistent over $\mathbb{F}_3$. Using Lemma \ref{scalar-extensions}, this inconsistency extends to any field of characteristic $3$. This includes the algebraically closed field $\overline{\mathbb{K}}$. Hence, none of the four even-part structures extends to a compatible left-symmetric superalgebra structure on $\mathfrak{osp}(1|2)$, thus contradicting the assumption we made earlier. Therefore, no compatible left-symmetric superalgebra structure exists on $\mathfrak{osp}(1|2)$ for $\mathrm{char}(\mathbb{K})$ different from $2$.
\end{proof}

During the computation process, we considered using SymPy's expression domain $\mathrm{EX}$ as an alternative field for the Gröbner basis algorithms. This domain is used for algebraic expressions that cannot be represented with traditional fields. In our case, the first two systems of equations have coefficients that depend on the arbitrary parameters $u, v \in \mathbb{K}$, and $\gamma \in \mathbb{K}^{\times}$ defined by Burde. Computing Gröbner bases over the domain $\mathrm{EX}$ would allow us to treat such coefficients as formal expressions rather than specializing the parameters to specific numerical elements. However, $\mathrm{EX}$ is a symbolic expression domain of characteristic $0$ like $\mathbb{Q}$ or $\mathbb{R}$, and is thus irrelevant to our discussion.

\subsubsection{Polynomial Factorization.} 

As a complementary approach to the Gröbner basis computation, we analyze the polynomial systems by factoring their defining equations. This method decomposes the algebraic variety of each system of equations into unions of subvarieties corresponding to vanishing factors. Consider the ideal $I_{1}$ for the first system, for example. We can express its algebraic variety as
    \begin{equation}
        V(I_{1}) = V(p_{1}) \cap \cdots \cap V(p_{203}). \label{variety-ideal1}
    \end{equation}
Say a polynomial $p_{k}$ factors as $p_{k} = h \cdot g$, where $h,g \in \mathbb{K}[u,v][\tau_{1},\ldots,\tau_{15}]$. Then, we can write $V(p_{k}) = V(h) \cup V(g)$. Substituting this into Eq. \eqref{variety-ideal1}, we may rewrite it as 
    \begingroup
    \small
    \begin{equation}
        V(I_{1}) = \underbrace{(V(p_{1}) \cap \cdots  \cap V(h) \cap \cdots \cap V(p_{203}))}_{\text{subvariety $1$}} \cup \underbrace{(V(p_{1}) \cap \cdots \cap V(g) \cap \cdots \cap V(p_{203}))}_{\text{subvariety $2$}}. \label{decomp_variety}
    \end{equation}
    \endgroup
    
\noindent
If a polynomial equation admits a nontrivial factorization, then any solution must lie in one of the algebraic subvarieties defined by the vanishing of one of its factors. The branch analysis consists of two sequential levels of inconsistency detection. The first level entails nonzero constant checks, where after substituting a linear factor into all 203 polynomials, the code searches for any remaining nonzero constant expressions. That is, the code looks for remaining contradictory expressions of the form $c = 0$ where $c \in \mathbb{K} \backslash \{0\}$. 

The second level deals with any polynomials that contain at least one unknown, even after the substitutions in level 1. This means that, even after level 1 factorization, we do not reach an obvious constant contradiction yet. However, the system may still be inconsistent due to interactions between the polynomials that survived. Thus, this code block takes these remaining polynomials, reduces them modulo $3$, and computes their Gröbner basis. If the resulting Gröbner basis is $G=\{1\}$, then by Section \ref{Grobner bases}, 1 belongs to the corresponding branch ideal, and thus the branch system of equations is inconsistent. The detailed approach explaining Eq. \eqref{decomp_variety} is described below. Consider the first system of equations in Eq. \eqref{sys-1} such that 
\begin{equation*}
    p_{1} = \cdots = p_{203} = 0.
\end{equation*}

\noindent
Here, any possible solution for a specific polynomial must satisfy every equation simultaneously. Now, suppose that a polynomial $p_{k}$, for some $k=1,\ldots,203$, factors nontrivially as a product of the following form
\begin{equation*}
    p_{k} = h \cdot g, 
\end{equation*}

\noindent
where $h,g \in \mathbb{K}[u,v][\tau_{1}, \ldots, \tau_{15}]$. Since $\mathbb{K}[u,v][\tau_{1}, \ldots, \tau_{15}]$ is an integral domain, the equation $h \cdot g = 0$ holds if and only if $h=0$ or $g=0$. Thus, any solution to the full system of polynomial equations must lie in one of two branches: 

\begin{enumerate}
    \setlength{\itemsep}{0pt}   % Space between items
    \setlength{\parskip}{0pt}   % Space between paragraphs in items
    \setlength{\topsep}{0pt}
    \item [(i)] $h = 0$ and $p_{1} = \cdots = p_{k-1} = p_{k+1} = \cdots = p_{203} = 0$,
    \item [(ii)] $g = 0$ and $p_{1} = \cdots = p_{k-1} = p_{k+1} = \cdots = p_{203} = 0$. 
\end{enumerate}

\noindent
If both branches are independently inconsistent, the entire system has no solution. If, however, we do not find inconsistencies in either of the branches, then we proceed to level 2 and compute the Gröbner basis for the corresponding branch ideals. The same logic follows for all other structures and their respective systems of equations. 

The factorization is performed over $\mathbb{Q}$ using SymPy's \verb|sp.factor()| function. This is a computational limitation, as the SymPy library does not support factorization of multivariate polynomials over finite fields, in this case $\mathbb{F}_{3}$. As a consequence, the irreducibility classification is a statement of irreducibility over $\mathbb{Q}$, not over $\mathbb{F}_{3}$. Factorization depends on the ground field; an irreducible polynomial over $\mathbb{Q}$ may further factor over $\mathbb{F}_{3}$, or conversely. Evidently, the factorization of these polynomials over $\mathbb{Q}$ does not determine their factorization over arbitrary fields of characteristic $3$. Nevertheless, the branch analysis conclusions below are still valuable for this paper for the following reasons. First, the zero-product property holds over any field. Second, the linear factors found for both structures are expressions with integer coefficients. Third, the contradictions produced by the branch analysis, that is, nonzero constants remaining after substitution, are verified to be nonzero modulo $3$. This means that they remain contradictions even in characteristic 3. Hence, although factorization may not capture all possible decompositions over $\mathbb{F}_{3}$, every branch identified is shown to be inconsistent modulo 3.

For structure $\mathcal{A}_{1,u,v}$, out of the 203 polynomials, 200 are irreducible, while the remaining 3 polynomials are scalar multiples of the following product
\[
    \alpha_{-+}^{F}(\beta_{+E}^{-}) = 0. 
\]
\noindent
These precisely correspond to polynomials $p_{16}, p_{104}$ and $p_{134}$. For each factored polynomial, this yields linear factors, which means $\alpha_{-+}^{F} = 0$ or $\beta_{+E}^{-} = 0$. Therefore, any solution must lie in one of the following branches:

\begin{enumerate}
    \setlength{\itemsep}{0pt}   % Space between items
    \setlength{\parskip}{0pt}   % Space between paragraphs in items
    \setlength{\topsep}{0pt}
    \item [(i)] $\alpha_{-+}^{F} = 0$ or,
    \item [(ii)] $\beta_{+E}^{-} = 0$. 
\end{enumerate}

\noindent
For each branch, the constraint is substituted into all 203 polynomial equations. In the first branch, letting $\alpha_{-+}^{F} = 0$ generates inconsistencies. Namely, assuming this equality forces polynomial $p_{80} = 1$, but this is an inconsistency because it would imply $1=0$. Hence, the system of polynomial equations corresponding to structure $\mathcal{A}_{1,u,v}$ has no solution. 

A similar result occurs with the system of equations for structure $\mathcal{A}_{2,\gamma}$. Of the $203$ polynomials, $153$ are irreducible. The remaining $50$ equations can be factored nontrivially, with $100$ polynomials with linear factors. From the $50$ factored polynomials, $20$ unique linear factors were found, which can be classified into: 

\begin{enumerate}
    \setlength{\itemsep}{0pt}   % Space between items
    \setlength{\parskip}{0pt}   % Space between paragraphs in items
    \setlength{\topsep}{0pt}
    \item [(i)] Single-variable factors: $\beta_{+E}^{-}=0, \ \alpha_{-+}^{F}=0, \ \beta_{-E}^{-}=0, \ \beta_{-E}^{+}=0, \ \alpha_{-+}^{E}=0, \ \beta_{+F}^{-} = 0, \ \beta_{-F}^{+} = 0, \ \beta_{+F}^{+} = 0, \ \alpha_{-+}^{H} = 0, \ \beta_{+H}^{-} =0, \ \beta_{-H}^{+} = 0$. These force the structure coefficients to be equal to zero;
    \item [(ii)] Sum or difference of relations: $\beta_{-E}^{-} + \beta_{+E}^{+} = 0, \ \beta_{-F}^{-} + \beta_{+F}^{+} = 0, \ \beta_{-H}^{-}\gamma^{-1} + \beta_{+H}^{+}\gamma^{-1} -2\gamma^{-1} -1 = 0, \ \beta_{-H}^{-}\gamma^{-1} + \beta_{+H}^{+}\gamma^{-1} +2\gamma^{-1} -1 = 0$. These force linear dependencies between structure constants;
    \item [(iii)] Fixed-value constants: $\alpha_{-+}^{H} - 1=0, \ \beta_{-H}^{-}\gamma^{-1} - \gamma^{-1} - 1 = 0, \ \beta_{-H}^{-}\gamma^{-1} - 1 = 0, \ \beta_{+H}^{+}\gamma^{-1} + \gamma^{-1} - 1 = 0, \ \beta_{+H}^{+}\gamma^{-1} -1 =0$. 
\end{enumerate}

\noindent
Some examples of these factored polynomials are $q_{10}, \ q_{16}$, and $q_{78}$, each representing one of the previous categories. These polynomials factor as follows
\begin{align*}
    q_{10} &= -\beta_{+E}^{-}(\beta_{-E}^{-} + \beta_{+E}^{+}), \\
    q_{16} &= -\alpha_{-+}^{F}(\beta_{+E}^{-}),\\
    q_{78} &= -\beta_{+H}^{-}(\beta_{-H}^{-} + \beta_{+H}^{+} - 2 -\gamma^{-1}).
\end{align*}
Using polynomial $q_{10}$ as an example for factors with a sum or difference relation, if we let $\beta_{+E}^{-} = 0$, no immediate contradiction appears. Nonetheless, we obtain a Gröbner basis of $G = \{1\}$ from the corresponding branch ideal. We obtain the same result letting $\beta_{-E}^{-} + \beta_{+E}^{+} = 0$. Using polynomial $q_{16}$ as an example for the single-variable factors, letting $\alpha_{-+}^{F} = 0$ or $\beta_{+E}^{-} = 0$, we obtain the same result as with polynomial $q_{10}$. Lastly, for polynomial $q_{78}$ as an example for fixed-value constants, setting $\beta_{+H}^{-} = 0$ or $\beta_{-H}^{-} + \beta_{+H}^{+} - 2 -\gamma^{-1} = 0$ also yields the same result. All $20$ linear factor branches are inconsistent, thus confirming that the system of equations corresponding to structure $\mathcal{A}_{2,\gamma}$ also has no solution. 

For the third system of equations, of the $203$ polynomials, $161$ are irreducible equations, while the remaining $42$ equations can be factored nontrivially. These can also be classified into the previous categories. Some examples of these factored polynomials are $r_{16}, \ r_{89}$, and $r_{93}$. These factor as follows
\begin{align*}
    r_{16} &= -\beta_{-E}^{-}(\beta_{+E}^{-}), \\
    r_{89} &= -\beta_{-H}^{+}(\beta_{-H}^{-} + \beta_{+H}^{+}), \\
    r_{93} &= \alpha_{-+}^{H}(\beta_{+H}^{+} + 1).
\end{align*}

\noindent
For all $42$ factored polynomials, we find no solutions as well. For polynomial $r_{16}$, we find no initial contradiction; nevertheless, we obtain a Gröbner basis $G = \{1\}$. For polynomial $r_{89}$, we find an immediate contradiction upon substituting either $\beta_{-H}^{+} = 0$ or $\beta_{-H}^{-} + \beta_{+H}^{+} = 0$ into the system of equations $\{r_i = 0\}_{i=1}^{203}$. Specifically, for polynomial $r_{94} = \beta_{-H}^{+}(\alpha_{-+}^{E}) + 1$, substituting $\beta_{-H}^{+} = 0$ forces the contradiction $1=0$. For polynomial $r_{93}$ we obtain a similar result: Substituting $\alpha_{-+}^{H} = 0$ into polynomial $r_{32} = \beta_{-E}^{+}(\alpha_{-+}^{H}) -1$ also yields the same inconsistency.

Lastly, the fourth system of equations yields similar results. Of the $203$ equations, $160$ are irreducible while the remaining $43$ can be factored nontrivially. In total, $18$ unique linear factors were found, of which the majority resulted in inconsistent branches with Gröbner bases $G = \{1\}$. We take polynomials $s_{78}, \ s_{84}$ and $s_{185}$ as examples, which factor as follows
\begin{align*}
    s_{78} &= \beta_{+H}^{-}(\xi - \beta_{-H}^{-} - \beta_{+H}^{+} + 2), \\
    s_{84} &= (\xi - \beta_{-H}^{-})(\alpha_{-+}^{H} -1), \\
    s_{185} &= \beta_{-+}^{E}(\xi - \beta_{-H}^{-} - \beta_{+H}^{+} -2). 
\end{align*}

\noindent
Factorization and substitution of such factors into the system of equations, or subsequent Gröbner basis computation, are also unsuccessful in this case. Thus, this system of equations is inconsistent. 

In practice, we found that the factorization yields no nontrivial decomposition leading to a consistent solution of each system. Specifically, either the polynomials remain irreducible over the chosen coefficient field, or their factors impose constraints that contradict other equations in the system. Specifically, no branch of the factorized system admits a simultaneous solution of all graded left-symmetric identities.

This factor-based analysis is consistent with the Gröbner basis computation done earlier. Although the Gröbner basis approach identifies global inconsistency, the factorization method analyzes the system locally by decomposing it into algebraic subcases. The fact that no factor branch yields a solution provides independent confirmation that the polynomial ideal generated by the graded left-symmetric identities is the entire polynomial ring. In other words, $I_1 = \mathbb{K}[u,v][\tau_1, \ldots, \tau_{15}]$, $I_2 = \mathbb{K}[\gamma, \gamma^{-1}][\tau_1, \ldots, \tau_{15}]$, $I_3 = \mathbb{K}[\tau_1, \ldots, \tau_{15}]$, and $I_4 = \mathbb{K}[\xi][\tau_1, \ldots, \tau_{15}]$.

Altogether, the Gröbner basis computation and the factorization-based approaches demonstrate that the left-symmetric identity admits no solution extending Burde's left-symmetric algebra structures from $\mathfrak{sl}_2(\mathbb{K})$ to $\mathfrak{osp}(1|2)$. In Stage $1$ of the superization, we found that the left product is compatible with the supercommutator of $\mathfrak{osp}(1|2)$. However, in Stage 2, we fail to reconcile the left product with the graded left-symmetric identity. The $203$ nontrivial polynomial equations from each algebra arise solely from associators containing $T_{\pm}$. Thus, the obstruction is entirely attributed to the interaction between $\mathfrak{sl}_{2}(\mathbb{K})$ and $\mathfrak{osp}(1|2)_{\bar{1}}$. This indicates that the obstruction is a result of structural incompatibility between the odd and even parts of $\mathfrak{osp}(1|2)$. Given that the obstruction already emerges at the level of the polynomial identities, this impossibility holds uniformly over all admissible values of the parameters defining the even left-symmetric structures. 

Further research could explore this obstruction through a cohomological lens analogous to Burde's use of $H^{1}(\mathfrak{g}, M_{\lambda})$ for $\mathfrak{sl}_{2}(\mathbb{K})$. Perhaps this obstruction could be explained through a superized cohomological framework. Moreover, Helmstetter's theorem \cite{helmstetter} could be further investigated, namely by exploring whether a similar theorem exists for the super case. That is, if $\mathfrak{g}$ is a simple Lie superalgebra over a field of characteristic $p \in \mathbb{Z}^{+}$ with $[\mathfrak{g}, \mathfrak{g}] = \mathfrak{g}$, under what circumstances does it fail to admit LSSAs? 

\newpage
\section{Appendix}
\subsection{Verification of Burde's Structures: The Left-Symmetric Identity.}\label{Appendix 1: Tables}

Below are the tables with all of the computations for the associators and their differences for the structures proposed by Burde in \cite{burde1}, as mentioned in Section \ref{Verif Burde LSAs}. Note that for structure $\mathcal{A}_{1,u,v}$, we included only the difference of associators as some of the associator expressions were very long and thus would overflow the document margins. The computations for all structures were performed using Python scripts \verb|lss_1_op.py|, \verb|lss_2_op.py|, \verb|lss_3_op.py|, and \verb|lss_4_op.py| found in the GitHub repository \href{https://github.com/jl13583/capstone_josemaria}{capstone\textunderscore josemaria}.

% =====================================================================
% Table for Structure 1 (LSS 1) — rows 30–56 of spreadsheet
% =====================================================================

\begin{table}[H]
\centering
\small
\setlength{\tabcolsep}{4pt}
\renewcommand{\arraystretch}{1.3}
\begin{tabular}{ccc}
\hline
$(X,Y,Z)$ & $(Y,X,Z)$ & $\Delta(X,Y,Z)$ \\ \hline
$(E,F,H)$ & $(F,E,H)$ & $-3vF = 0$ \\ 
$(H,F,E)$ & $(F,H,E)$ & $-3uE + (3u - 3v)F + 6u(-u + v)H = 0$ \\ 
$(H,E,F)$ & $(E,H,F)$ & $-3uE + 3uF + (-6u^2 + 6uv - 3)H = 0$ \\ 
$(F,H,E)$ & $(H,F,E)$ & $3uE + (-3u + 3v)F + 6u(u - v)H = 0$ \\ 
$(E,E,E)$ & $(E,E,E)$ & $0$ \\ 
$(E,H,E)$ & $(H,E,E)$ & $3H = 0$ \\ 
$(E,F,E)$ & $(F,E,E)$ & $-3vH = 0$ \\ 
$(H,H,E)$ & $(H,H,E)$ & $0$ \\ 
$(F,F,E)$ & $(F,F,E)$ & $0$ \\ 
$(E,E,H)$ & $(E,E,H)$ & $0$ \\ 
$(E,H,H)$ & $(H,E,H)$ & $3uH = 0$ \\ 
$(H,H,H)$ & $(H,H,H)$ & $0$ \\ 
$(H,F,H)$ & $(F,H,H)$ & $-3v^2E + (3u + 3v)H = 0$ \\ 
$(F,F,H)$ & $(F,F,H)$ & $0$ \\ 
$(E,E,F)$ & $(E,E,F)$ & $0$ \\ 
$(E,F,F)$ & $(F,E,F)$ & $3v(-u + v)E + (-3u^2v + 6uv^2 - 3u - 3v^3 - 3v)H = 0$ \\ 
$(H,H,F)$ & $(H,H,F)$ & $0$ \\
$(H,F,F)$ & $(F,H,F)$ & $3u(uv - v^2 - 1)E - 3vF + 6v^2H = 0$ \\
$(F,F,F)$ & $(F,F,F)$ & $0$ \\
$(E,H,F)$ & $(H,E,F)$ & $3uE - 3uF + (6u^2 - 6uv + 3)H = 0$ \\
$(H,E,E)$ & $(E,H,E)$ & $-3H = 0$ \\
$(H,E,H)$ & $(E,H,H)$ & $-3uH = 0$ \\
$(F,E,E)$ & $(E,F,E)$ & $3vH = 0$ \\
$(F,E,H)$ & $(E,F,H)$ & $3vF = 0$ \\
$(F,E,F)$ & $(E,F,F)$ & $3v(u - v)E + (3u^2v - 6uv^2 + 3u + 3v^3 + 3v)H = 0$ \\
$(F,H,H)$ & $(H,F,H)$ & $3v^2E + (-3u - 3v)H = 0$ \\ 
$(F,H,F)$ & $(H,F,F)$ & $3u(-uv + v^2 + 1)E + 3vF + (-6v^2)H = 0$ \\ \hline
\end{tabular}
\caption{Associator verification for Burde's left-symmetric structure $\mathcal{A}_{1, u,v}$ on $\mathfrak{sl}_2(\mathbb{K})$. Note that the column $\Delta(X,Y,Z) = (X,Y,Z) - (Y,X,Z)$ confirms the left-symmetric identity, as all entries equal zero in characteristic $3$.}
\label{assoc_table_lss1}
\end{table}

% =====================================================================
% Table for Structure 2 (LSS 2) — rows 2–28 of spreadsheet
% =====================================================================

\begin{table}[H]
\centering
\small
\setlength{\tabcolsep}{4pt}
\renewcommand{\arraystretch}{1.3}
\begin{tabular}{ccccc}
\hline
$(X,Y,Z)$ & $(Y,X,Z)$ & $\Delta(X,Y,Z)$ & $(X,Y,Z)$ & $(Y,X,Z)$ \\ \hline
$(E,F,H)$ & $(F,E,H)$ & $0$ & $(-1 - \gamma^{-1})\gamma H$ & $(-1 - \gamma^{-1})\gamma H$ \\ 
$(H,F,E)$ & $(F,H,E)$ & $0$ & $H - \gamma^{-1}H$ & $H - \gamma^{-1}H$ \\
$(H,E,F)$ & $(E,H,F)$ & $0$ & $H + \gamma^{-1}H$ & $H + \gamma^{-1}H$ \\ 
$(F,H,E)$ & $(H,F,E)$ & $0$ & $H - \gamma^{-1}H$ & $H - \gamma^{-1}H$ \\ 
$(E,E,E)$ & $(E,E,E)$ & $0$ & $0$ & $0$ \\ 
$(E,H,E)$ & $(H,E,E)$ & $0$ & $0$ & $0$ \\ 
$(E,F,E)$ & $(F,E,E)$ & $3E - 3\gamma E = 0$ & $E - 2\gamma E + \gamma^{-1}E$ & $\gamma E - 2E + \gamma^{-1}E$ \\ 
$(H,H,E)$ & $(H,H,E)$ & $0$ & $\gamma E - E$ & $\gamma E - E$ \\
$(F,F,E)$ & $(F,F,E)$ & $0$ & $F - \gamma F$ & $F - \gamma F$ \\ 
$(E,E,H)$ & $(E,E,H)$ & $0$ & $0$ & $0$ \\ 
$(E,H,H)$ & $(H,E,H)$ & $0$ & $0$ & $0$ \\ 
$(H,H,H)$ & $(H,H,H)$ & $0$ & $0$ & $0$ \\ 
$(H,F,H)$ & $(F,H,H)$ & $0$ & $0$ & $0$ \\ 
$(F,F,H)$ & $(F,F,H)$ & $0$ & $0$ & $0$ \\ 
$(E,E,F)$ & $(E,E,F)$ & $0$ & $\gamma E + E$ & $\gamma E + E$ \\
$(E,F,F)$ & $(F,E,F)$ & $-(3\gamma F + 3F) = 0$ & $-\gamma F - 2F - \gamma^{-1}F$ & $2\gamma F - \gamma^{-1}F + F$ \\
$(H,H,F)$ & $(H,H,F)$ & $0$ & $-\gamma F - F$ & $-\gamma F - F$ \\
$(H,F,F)$ & $(F,H,F)$ & $0$ & $0$ & $0$ \\ 
$(F,F,F)$ & $(F,F,F)$ & $0$ & $0$ & $0$ \\ 
$(E,H,F)$ & $(H,E,F)$ & $0$ & $H + \gamma^{-1}H$ & $H + \gamma^{-1}H$ \\ 
$(H,E,E)$ & $(E,H,E)$ & $0$ & $0$ & $0$ \\ 
$(H,E,H)$ & $(E,H,H)$ & $0$ & $0$ & $0$ \\ 
$(F,E,E)$ & $(E,F,E)$ & $3\gamma E - 3E = 0$ & $\gamma E - 2E + \gamma^{-1}E$ & $-2\gamma E + \gamma^{-1}E + E$ \\ 
$(F,E,H)$ & $(E,F,H)$ & $0$ & $(-1 - \gamma^{-1})\gamma H$ & $(-1 - \gamma^{-1})\gamma H$ \\
$(F,E,F)$ & $(E,F,F)$ & $3\gamma F + 3F = 0$ & $2\gamma F - \gamma^{-1}F + F$ & $-\gamma F - 2F - \gamma^{-1}F$ \\ 
$(F,H,H)$ & $(H,F,H)$ & $0$ & $0$ & $0$ \\
$(F,H,F)$ & $(H,F,F)$ & $0$ & $0$ & $0$ \\ \hline
\end{tabular}
\caption{Associator verification for Burde's left-symmetric structure $\mathcal{A}_{2, \gamma}$ on $\mathfrak{sl}_2(\mathbb{K})$. Note that the column $\Delta(X,Y,Z) = (X,Y,Z) - (Y,X,Z)$ confirms the left-symmetric identity, as all entries equal zero in characteristic $3$.}
\label{assoc_table_lss2}
\end{table}

% =====================================================================
% Table for Structure 3 (LSS 3)
% =====================================================================

\begin{table}[H]
\centering
\small
\setlength{\tabcolsep}{4pt}
\renewcommand{\arraystretch}{1.3}
\begin{tabular}{ccccc}
\hline
$(X,Y,Z)$ & $(Y,X,Z)$ & $\Delta(X,Y,Z)$ & $(X,Y,Z)$ & $(Y,X,Z)$ \\ \hline
$(E,F,H)$ & $(F,E,H)$ & $0$ & $0$ & $0$ \\ 
$(H,F,E)$ & $(F,H,E)$ & $-3H = 0$ & $-H$ & $2H$ \\
$(H,E,F)$ & $(E,H,F)$ & $0$ & $0$ & $0$ \\ 
$(F,H,E)$ & $(H,F,E)$ & $3H = 0$ & $2H$ & $-H$ \\ 
$(E,E,E)$ & $(E,E,E)$ & $0$ & $0$ & $0$ \\
$(E,H,E)$ & $(H,E,E)$ & $0$ & $0$ & $0$ \\ 
$(E,F,E)$ & $(F,E,E)$ & $0$ & $-E$ & $-E$ \\
$(H,H,E)$ & $(H,H,E)$ & $0$ & $-2E$ & $-2E$ \\
$(F,F,E)$ & $(F,F,E)$ & $0$ & $-F$ & $-F$ \\
$(E,E,H)$ & $(E,E,H)$ & $0$ & $0$ & $0$ \\ 
$(E,H,H)$ & $(H,E,H)$ & $0$ & $0$ & $0$ \\ 
$(H,H,H)$ & $(H,H,H)$ & $0$ & $0$ & $0$ \\ 
$(H,F,H)$ & $(F,H,H)$ & $0$ & $-E$ & $-E$ \\
$(F,F,H)$ & $(F,F,H)$ & $0$ & $H$ & $H$ \\ 
$(E,E,F)$ & $(E,E,F)$ & $0$ & $0$ & $0$ \\ 
$(E,F,F)$ & $(F,E,F)$ & $0$ & $-E$ & $-E$ \\
$(H,H,F)$ & $(H,H,F)$ & $0$ & $2H$ & $2H$ \\
$(H,F,F)$ & $(F,H,F)$ & $-3H = 0$ & $-E - H$ & $-E + 2H$ \\
$(F,F,F)$ & $(F,F,F)$ & $0$ & $-F + H$ & $-F + H$ \\
$(E,H,F)$ & $(H,E,F)$ & $0$ & $0$ & $0$ \\
$(H,E,E)$ & $(E,H,E)$ & $0$ & $0$ & $0$ \\ 
$(H,E,H)$ & $(E,H,H)$ & $0$ & $0$ & $0$ \\ 
$(F,E,E)$ & $(E,F,E)$ & $0$ & $-E$ & $-E$ \\
$(F,E,H)$ & $(E,F,H)$ & $0$ & $0$ & $0$ \\ 
$(F,E,F)$ & $(E,F,F)$ & $0$ & $-E$ & $-E$ \\
$(F,H,H)$ & $(H,F,H)$ & $0$ & $-E$ & $-E$ \\ 
$(F,H,F)$ & $(H,F,F)$ & $-3H = 0$ & $-E - H$ & $-E + 2H$ \\ \hline
\end{tabular}
\caption{Associator verification for Burde's left-symmetric structure $\mathcal{A}_{3}$ on $\mathfrak{sl}_2(\mathbb{K})$. The column $\Delta(X,Y,Z) = (X,Y,Z) - (Y,X,Z)$ confirms the left-symmetric identity, as all entries equal zero in characteristic $3$.}
\label{assoc_table_lss3}
\end{table}

% =====================================================================
% Table for Structure 4 (LSS 4)
% =====================================================================

\begin{table}[H]
\centering
\small
\setlength{\tabcolsep}{4pt}
\renewcommand{\arraystretch}{1.3}
\begin{tabular}{ccccc}
\hline
$(X,Y,Z)$ & $(Y,X,Z)$ & $\Delta(X,Y,Z)$ & $(X,Y,Z)$ & $(Y,X,Z)$ \\ \hline
$(E,F,H)$ & $(F,E,H)$ & $0$ & $0$ & $0$ \\
$(H,F,E)$ & $(F,H,E)$ & $0$ & $(\xi + 1)H$ & $(\xi + 1)H$ \\ 
$(H,E,F)$ & $(E,H,F)$ & $0$ & $(1 - \xi)H$ & $(1 - \xi)H$ \\ 
$(F,H,E)$ & $(H,F,E)$ & $0$ & $(\xi + 1)H$ & $(\xi + 1)H$ \\ 
$(E,E,E)$ & $(E,E,E)$ & $0$ & $0$ & $0$ \\ 
$(E,H,E)$ & $(H,E,E)$ & $0$ & $0$ & $0$ \\ 
$(E,F,E)$ & $(F,E,E)$ & $(-\xi^2 - 3\xi + 2)E = 0$ & $(1 - 3\xi)E$ & $(\xi^2 - 1)E$ \\ 
$(H,H,E)$ & $(H,H,E)$ & $0$ & $(\xi - 1)E$ & $(\xi - 1)E$ \\ 
$(F,F,E)$ & $(F,F,E)$ & $0$ & $-\xi(\xi + 1)F$ & $-\xi(\xi + 1)F$ \\ 
$(E,E,H)$ & $(E,E,H)$ & $0$ & $0$ & $0$ \\ 
$(E,H,H)$ & $(H,E,H)$ & $0$ & $0$ & $0$ \\ 
$(H,H,H)$ & $(H,H,H)$ & $0$ & $0$ & $0$ \\ 
$(H,F,H)$ & $(F,H,H)$ & $0$ & $0$ & $0$ \\ 
$(F,F,H)$ & $(F,F,H)$ & $0$ & $0$ & $0$ \\ 
$(E,E,F)$ & $(E,E,F)$ & $0$ & $\xi(1 - \xi)E$ & $\xi(1 - \xi)E$ \\
$(E,F,F)$ & $(F,E,F)$ & $(\xi^2 - 3\xi - 2)F = 0$ & $(\xi^2 - 1)F$ & $(3\xi + 1)F$ \\ 
$(H,H,F)$ & $(H,H,F)$ & $0$ & $-(\xi + 1)F$ & $-(\xi + 1)F$ \\ 
$(H,F,F)$ & $(F,H,F)$ & $3E = 0$ & $2E$ & $-E$ \\ 
$(F,F,F)$ & $(F,F,F)$ & $0$ & $-2H$ & $-2H$ \\ 
$(E,H,F)$ & $(H,E,F)$ & $0$ & $(1 - \xi)H$ & $(1 - \xi)H$ \\
$(H,E,E)$ & $(E,H,E)$ & $0$ & $0$ & $0$ \\ 
$(H,E,H)$ & $(E,H,H)$ & $0$ & $0$ & $0$ \\ 
$(F,E,E)$ & $(E,F,E)$ & $(\xi^2 + 3\xi - 2)E = 0$ & $(\xi^2 - 1)E$ & $(1 - 3\xi)E$ \\ 
$(F,E,H)$ & $(E,F,H)$ & $0$ & $0$ & $0$ \\ 
$(F,E,F)$ & $(E,F,F)$ & $(-\xi^2 + 3\xi + 2)F = 0$ & $(3\xi + 1)F$ & $(\xi^2 - 1)F$ \\ 
$(F,H,H)$ & $(H,F,H)$ & $0$ & $0$ & $0$ \\ 
$(F,H,F)$ & $(H,F,F)$ & $-3E = 0$ & $-E$ & $2E$ \\ \hline
\end{tabular}
\caption{Associator verification for Burde's left-symmetric structure $\mathcal{A}_{4}$ on $\mathfrak{sl}_2(\mathbb{K})$, where $\xi^2 = -1$. The column $\Delta(X,Y,Z) = (X,Y,Z) - (Y,X,Z)$ confirms the left-symmetric identity, as all entries equal zero in characteristic $3$ using $\xi^2 = -1$.}
\label{assoc_table_lss4}
\end{table}

\noindent {\bf Acknowledgements.} The author gratefully acknowledges that this work was carried out under the supervision of Professor Sofiane Bouarroudj  and sincerely thanks him for his guidance and valuable discussions throughout the project.

\bibliographystyle{amsplain}

\bibliography{references}

\end{document}